\documentclass[12pt]{amsart}

\usepackage{color}
\usepackage{amsmath,amssymb,amsthm,mathtools}
\usepackage{wrapfig}
\usepackage{ifpdf}
\usepackage{graphicx,subcaption}
\usepackage{tikz}
\usetikzlibrary{positioning}

\usepackage{hyperref}
\hypersetup{colorlinks}

\newtheorem{theorem}{Theorem}[section]
\newtheorem{cor}[theorem]{Corollary}
\newtheorem{lemma}[theorem]{Lemma}
\newtheorem{prop}[theorem]{Proposition}
\newtheorem{defn}[theorem]{Definition}
\newtheorem{remark}[theorem]{Remark}
\newtheorem{exa}[theorem]{Example}
\newtheorem{que}[theorem]{Question}

\title {Circle actions on oriented 4-manifolds} 

\author{Donghoon Jang}
\address{Department of Mathematics, Pusan National University, 2, Busandaehak-ro 63beon-gil, Geumjeong-gu, Busan, 46241, Republic of Korea.}
\email{donghoonjang@pusan.ac.kr}

\author{Oleg R. Musin}
\address{School of Mathematical and Statistical Sciences, University of Texas Rio Grande Valley, One West University Boulevard, Brownsville, TX, 78520, USA.} 
\email{oleg.musin@utrgv.edu}

\begin{document}

\maketitle

\begin{abstract}
In the present paper, we consider an action of the circle group on a compact oriented 4-manifold. We derive the Atiyah-Hirzebruch formula for the manifold, and associate a graph in terms of data on the fixed point set. We show in the case of isolated fixed points that if an abstract graph satisfies the Atiyah-Hirzebruch formula, then there exists a corresponding 4-dimensional oriented $S^1$-manifold.
\end{abstract}

\tableofcontents

\section{Introduction}


The problem of classification of torus and circle actions on 4-manifolds was considered in 1970's by Orlik and  Raymond \cite{OR1, OR2}, Fintushel \cite{F1, F2, F3}, Pao \cite{pao77a, pao77b}, and Yoshida \cite{yos78}, and later by Melvin \cite{mel81}, Melvin and Parker \cite{MP}, Huck \cite{huc95}, Huck and Puppe \cite{HP}, etc. 
We refer to \cite{edm09} for a survey on group actions on 4-manifolds.
Circle actions on different types of 4-manifolds have been also studied; Carrell, Howard, and Kosniowski \cite{CHK} studied complex manifolds, Ahara and Hattori \cite{AH}, Audin \cite{Au}, and Karshon \cite{Ka} studied symplectic manifolds, and the first named author studied almost complex manifolds \cite{jan19b}.

Orlik and Raymond proved that a $T^2$-action on a closed simply connected oriented 4-manifold is an equivariant connected sum of copies of $\pm \mathbb{CP}^2$ and $S^2$-bundles over $S^2$ \cite{OR1}. Fintushel \cite{F2} and Yoshida \cite{yos78} showed that an $S^1$-action on a closed simply connected 4-manifold is a connected sum of a homotopy $S^4$ and copies of $\pm \mathbb{CP}^2$ and $S^2 \times S^2$; their connected sums are not equivariant in general. Orlik and Raymond showed that a $T^2$-action on a closed orientable 4-manifold is determined by its orbit data \cite{OR2}; Fintushel proved a similar result for $S^1$-actions \cite{F3}.

Let $M$ be a 4-dimensional compact oriented $S^1$-manifold. After reviewing necessary background in Section \ref{s2}, in Section \ref{s3} we derive the Atiyah-Hirzebruch formula for $M$ (Theorem \ref{AHformula}), that is, the Atiyah-Singer index formula for the equivariant index of the signature operator on $M$. The formula is expressed in terms of signs and weights at isolated fixed points and the Euler numbers of the normal bundles of fixed surfaces. We discuss several consequences of Theorem \ref{AHformula}.

In Section \ref{s4}, using the Atiyah-Hirzebruch formula, in Theorem \ref{graph} we associate a certain type of graph to $M$ that we call a graph of weights (Definitions \ref{defngraph} and \ref{graphofweights}). The graph contains information on the data on its fixed point set, signs and weights at isolated fixed points and the Euler numbers of the normal bundles of fixed surfaces. Moreover, any edge with label $w$ bigger than 1 corresponds to an invariant 2-sphere of weight $w$ containing two fixed points that are the vertices of the edge; see (2) of Theorem \ref{graph}.
We note that a graph of weights does not determine an $S^1$-manifold uniquely; see Remark \ref{non-uniqueness}.

In Section \ref{s5}, we discuss equivariant connected sum and splitting. A traditional connected sum of two $m$-manifolds removes a ball of each manifold and glue the boundary spheres. 
We will also take a connected sum along neighborhoods of submanifolds, and also call its reverse operation a splitting.
Utilizing the graph that we associate to $M$, in Theorem \ref{ReduceWeight} we show that we can equivariantly split $M$ into $M_0$ and copies of $\pm \mathbb{CP}^2$, where $M_0$ is another $S^1$-manifold which is minimal in the sense that every weight in the normal bundle of any fixed component is 1. Therefore, for a further topological classification, it is natural to ask if there also exist invariant 2-spheres for weight 1, see Question \ref{quesphere}; one may ask the same question for higher dimensional oriented $S^1$-manifolds.

In Section \ref{s6}, we prove a partial converse to a combination of Theorem \ref{AHformula} and Theorem \ref{graph}, that if an abstract graph of weights whose vertices all correspond to isolated fixed points (see Definition \ref{defngraph}) satisfies the Atiyah-Hirzebruch formula (Theorem \ref{AHformula}), then there exists a corresponding 4-dimensional compact oriented $S^1$-manifold; see Theorem \ref{graphconverse}.

\section*{Acknowledgements}

The authors would like to thank Mikiya Masuda for fruitful comments.
Donghoon Jang was supported by the National Research Foundation of Korea(NRF) grant funded by the Korea government(MSIT) (2021R1C1C1004158).

\section{Background and preliminaries} \label{s2}

Let a group $G$ act on a manifold $M$. Throughout this paper, any group action on a manifold is assumed to be effective. We denote its fixed point set by $M^G$, that is,
\begin{center}
$M^G=\{m \in M \, | \, g \cdot m=m \textrm{ for all }g \in G\}$.
\end{center}
If $H$ is a subgroup of $G$, $H$ also acts on $M$, and its fixed point set $M^H$ is defined analogously.

Let the circle group $S^1$ act on a $2n$-dimensional oriented manifold $M$. Let $F$ be a fixed component of dimension $2m$, and let $q$ be a point in $F$. The normal space $N_qF$ of $F$ to $M$ at $q$ decomposes into the sum of 2-dimensional irreducibles
$$N_qF=L_{F,1}\oplus \cdots \oplus L_{F,n-m},$$
where on each $L_{F,i}$ the circle acts by multiplication by $g^{w_{F,i}}$ for some non-zero integer $w_{F,i}$. These integers $w_{F,i}$ are the same for all $q \in F$ and called the \textbf{weights} of $F$. Throughout this paper, we choose an orientation of $L_{F,i}$ so that $w_{F,i}$ is positive for each $i$. The choice of the orientation of each $L_{F,i}$ then determines an orientation of $F$.

Let $p$ be an isolated fixed point. Then the tangent space at $p$ has two orientations; one induced from the orientation of $M$ and the other that we chose on $L_{p,1}\oplus \cdots \oplus L_{p,n}$. We define the \textbf{sign} of $p$, denoted by $\varepsilon_p$, to be $+1$ if the two orientations agree and $-1$ otherwise.

Let the circle group $S^1$ act on a manifold $M$. As a subgroup, the cyclic subgroup $\mathbb{Z}_w$ also acts on $M$. Its fixed point set $M^{\mathbb{Z}_w}$ is a submanifold of $M$.

Suppose that $M$ is orientable and $w>2$. Let $F$ be a component of $M^{\mathbb{Z}_w}$. Let $p \in F$. The normal space $N_pF$ of $F$ decomposes into irreducible $\mathbb{Z}_w$-representations $L_1 \oplus \cdots \oplus L_m$. On each $L_i$, the group $\mathbb{Z}_w$ acts as multiplication by a root of unity, and thus each $L_i$ has (real) dimension 2 and has an orientation induced and preserved by the group $\mathbb{Z}_w$. Since this holds for all points in $F$, the normal bundle of $F$ is orientable. Because $M$ is orientable, it follows that $F$ is also orientable. For references, see \cite[Lemma 12]{Kol}, \cite[Theorem 3.5.2]{Z}.

\begin{theorem} \label{l25} 
Let the circle group $S^1$ act effectively on an oriented manifold $M$. Let $w \geq 3$ be an integer. Then the $\mathbb{Z}_w$-fixed point set $M^{\mathbb{Z}_w}$ is orientable.
\end{theorem}

If $\dim M=4$, then a component of the $\mathbb{Z}_2$-fixed point set containing an $S^1$-fixed point is also orientable.

\begin{lemma} \label{Z2fxd}
Let $M$ be a 4-dimensional compact oriented $S^1$-manifold. Suppose that an isolated fixed point $p$ has a weight 2. Then the component of $M^{\mathbb{Z}_2}$ containing $p$ is orientable.
\end{lemma}

\begin{proof}
Let $F$ denote the component of $M^{\mathbb{Z}_2}$ that contains $p$. Assume on the contrary that $F$ is not orientable. The component $F$ is a compact submanifold of $M$, and the $S^1$-action on $M$ acts on $F$ as a restriction with a fixed point $p$. This implies that $\dim F=2$. As shown in \cite{Au2}, by the classification of $S^1$-actions on 2-dimensional compact manifolds, these imply that $F$ is $\mathbb{RP}^2$; moreover, the action of $S^1/\mathbb{Z}_2$ on $F$ has $Z:=\mathbb{RP}^1$ as a component of the $\mathbb{Z}_2$-fixed point set, and the normal bundle $N_1$ of $Z$ in $F$ is isomorphic to the canonical line bundle over $Z=\mathbb{RP}^1$, that is, a neighborhood of $Z$ in $F$ is diffeomorphic to an open M\"obius strip. Therefore, $Z$ is a $\mathbb{Z}_4$-fixed component of $M$. The normal bundle of $Z$ in $M$ splits into the Whitney sum of the normal bundle $N_1$ of $Z$ in $F$ and a real 2-dimensional $\mathbb{Z}_4$-vector bundle $N_2$ over $Z$, on which $\mathbb{Z}_4$ acts on fibers freely outside the zero section. Then $N_1$ is non-orientable while $N_2$ is orientable because the $\mathbb{Z}_4$-action on fibers makes $N_2$ a complex line bundle. This contradicts the orientability of $M$.
\end{proof}

A fixed point with weight $w>1$ lies in an invariant 2-sphere with weight $w$, and the 2-sphere contains another $S^1$-fixed point.

\begin{lemma} \label{isotropysphere}
Let $M$ be a 4-dimensional compact oriented $S^1$-manifold. Let $p$ be an isolated fixed point and let $w>1$ be a weight at $p$. Then the component of $M^{\mathbb{Z}_w}$ containing $p$ is a 2-sphere, invariant under the $S^1$-action, and it contains another fixed point $q$ that has a weight $w$.
\end{lemma}

\begin{proof}
Let $F$ be the component of $M^{\mathbb{Z}_w}$ that contains $p$. Because $M$ is compact, the component $F$ is a closed submanifold of $M$. Since the action is effective, two weights at $p$ are relatively prime and $F$ has dimension 2. By Theorem \ref{l25} (for $w>2$) and Lemma \ref{Z2fxd} (for $w=2$), $F$ is orientable. The $S^1$-action on $M$ restricts to an action on this 2-dimensional closed orientable manifold $F$ and has a fixed point $p$. This implies that $F$ is a 2-sphere and has another fixed point $q$ with a weight $w$.
\end{proof}

\section{Atiyah--Hirzebruch formula in dimension 4} \label{s3}

\subsection{Rigidity of genera}

By rigidity of a genus of a manifold we mean that if a group (torus) acts on a certain type of (compact) manifold, then its equivariant genus is equal to its ordinary genus; in particular, the equivariant genus is a constant.

Atiyah and Hirzebruch first proved such a rigidity result that if the circle group acts on a compact (oriented) spin manifol, then its equivariant $\hat{A}$-genus vanishes, by using the Atiyah-Singer index theorem \cite{AHr}. Krichever used the equivariant cobordism theory to prove that for a unitary $S^1$-manifold, the Hirzebruch $T_{x,y}$-genus is rigid \cite{kri74}. Rigidity of genera has been studied extensively. We refer to \cite{BPR} for a summury of the history on the theory of equivariant genera.

\subsection{Oriented manifold, dimension 4}

For a compact oriented manifold, the $L$-genus is the genus of the power series $\frac{\sqrt{x}}{\tanh \sqrt{x}}$. The Hirzebruch signature theorem states that for a compact oriented manifold $M$, its $L$-genus evaluated on the fundamental class of $M$ is equal to the signature of $M$. The Atiyah-Signer index theorem states that the signature of $M$ is equal to the (analytic) index of the signature operator on $M$ \cite{AS}.

Let the circle act on a $2n$-dimensional compact oriented manifold $M$. For each element $z$ of the circle group $S^1$, its equivariant index of the signature operator on $M$ is defined. Let $F$ be a fixed component. Let $\dim F=2m$ and $TM|_F=TF \oplus L_1 \oplus \cdots \oplus L_{n-m}$, where the circle acts on each $L_j$ with weight $w_j$. Denote by $c_1(L_j)$ the first Chern class of $L_j$, $P(F)=\prod_{i=1}^m (1+x_i^2)$ the total Pontryagin class of $F$, and $[F]$ the fundamental homology class of $F$. 

The signature (in other words, the $L$-genus) of a compact oriented $S^1$-manifold is rigid; moreover, the following formula holds for the equivariant signature.

\begin{theorem}\label{EquiSign} \cite[p. 72]{HBJ} Let the circle act on a $2n$-dimensional compact oriented manifold $M$. The equivariant signature $\textrm{sign}(z,M)$ of $M$ is
\begin{center}
$\displaystyle \mathrm{sign}(M)=\mathrm{sign}(z,M)=\sum_{F \subset M^{S^1}} \left\{ \left( \prod_{i=1}^m x_i \frac{1+e^{-x_i}}{1-e^{-x_i}} \right) \left( \prod_{j=1}^{n-m} \frac{1+z^{w_j}e^{-c_1(L_j)}}{1-z^{w_j}e^{-c_1(L_j)}} \right) \right\} [F]$
\end{center}
for all indeterminates $z$, and is a constant.
\end{theorem}

Note that the signature is the only rigid Hirzebruch genera for oriented manifolds \cite[Theorem 4.2]{mus11}.

Now let $M$ have dimension 4. We label isolated fixed points by $p_1$, $\cdots$, $p_m$ and fixed surfaces by $F_1$, $\cdots$, $F_k$. Let $w_{i1}$ and $w_{i2}$ denote the weights at $p_i$, for $1 \leq i \leq m$. Let $n_j$ denote the Euler number of the normal bundle of $F_j$, for $1 \leq j \leq k$. With these notations, we derive the Atiyah-Hirzebruch formula for a 4-dimensional compact oriented $S^1$-manifold.

\begin{theorem}\label{AHformula}
Let $S^1$ act on a 4-dimensional compact oriented manifold $M$. The signature $\mathrm{sign}(M)$ of $M$ satisfies
\begin{equation} \label{ASformula}
\displaystyle 
\mathrm{sign}(M)=\sum\limits_{i=1}^m\varepsilon_i\, {\frac{(1+z^{w_{i1}})(1+z^{w_{i2}})}{(1-z^{w_{i1}})(1-z^{w_{i2}})}}-\sum\limits_{j=1}^k\frac{4zn_j}{(1-z)^2}=\sum\limits_{i=1}^m\varepsilon_i
\end{equation}
for all indeterminates $z$, and
$$
\mathrm{sign}(M)=\frac{1}{3}p_1(M),
$$
where $p_1(M)$ is the first Pontryagin class of $M$.
\end{theorem}

\begin{proof}
First, we consider an isolated fixed point $p_i$. Its total Pontryagin class is $P(p_i)=1$ and $c_1(L_j)=0$ for $j=1,2$. Thus, in the formula of Theorem \ref{EquiSign} the fixed point $p_i$ contributes the term
$$
\varepsilon_i\, {\frac{(1+z^{w_{i1}})(1+z^{w_{i2}})}{(1-z^{w_{i1}})(1-z^{w_{i2}})}}.
$$

Let $F_j$ be a fixed surface. Let $P(F_j)=1+x^2$ be the total Pontryagin class of $F_j$. Then
\begin{center}
$\displaystyle x \frac{1+e^{-x}}{1-e^{-x}}=x \frac{2 - x + \cdots}{x - \frac{x^2}{2} + \cdots} = \frac{2-x + \cdots}{1- \frac{x}{2} + \cdots}=(2-x + \cdots) \left(1+\frac{x}{2}+\cdots \right)=2+x-x+\cdots=2+ 0 \cdot x+ O(x^2).$
\end{center}

Next, let $L$ denote the normal bundle of $F_j$, $u \in H^2(F_j)$ a generator, and $c_1(L)=n_j u$. The weight $w_j$ of $F_j$ is 1. Then
$$
\frac{1+z^{w_j}e^{-c_1(L_j)}}{1-z^{w_j}e^{-c_1(L_j)}}=\frac{1+ze^{-n_ju}}{1-ze^{-n_ju}}=\frac{(1+z)+z(-n_ju)+\cdots}{(1-z)+zn_ju+\cdots}.
$$

Using $(1-z)+zn_ju+\cdots=(1-z)(1-(-\frac{zn_ju}{1-z})+\cdots)$, this is equal to
\begin{center}
$\displaystyle \frac{1}{1-z} \cdot \frac{(1+z)-zn_ju+\cdots}{1-(-\frac{zn_ju}{1-z})+\cdots}=\frac{1+z}{1-z}+\left(\frac{-zn_ju}{1-z}+\frac{1+z}{1-z}\cdot(-\frac{zn_ju}{1-z})\right)+\cdots=\frac{1+z}{1-z}+\frac{-2n_jz}{(1-z)^2}u +\cdots.$
\end{center}
Thus,
$$
\left\{ x \frac{1+e^{-x}}{1-e^{-x}} \cdot \frac{1+z e^{-c_1(L_j)}}{1-z e^{-c_1(L_j)}} \right\} [F_j]=\frac{-4zn_j}{(1-z)^2}.
$$
Therefore, the first equation follows. Taking $z=0$, the second equation holds. The last equation $\mathrm{sign}(M)=\frac{1}{3} p_1(M)$ follows by the Hirzebruch signature theorem.
\end{proof}

Note that the Atiyah-Hirzebruch formula above does not see the genus of a fixed surface. Also, note that the Atiyah-Hirzebruch formula is derived when a finite group of odd order acts on a compact oriented 4-manifold \cite[Proposition 6.18, p. 585]{AS} and \cite[p. 176]{HZ}. Using the above formula, we obtain the following formulas.

\begin{theorem} \label{3formulas}
Let $S^1$ act on a 4-dimensional compact oriented manifold $M$. The following equations hold.
\begin{enumerate}
\item \begin{center}
$\displaystyle L(M)=\sum_{i=1}^m \varepsilon_i \frac{w_{i1}^2+w_{i2}^2+1}{3w_{i1}w_{i2}}.$
\end{center}
\item \begin{center}
$\displaystyle 0=\sum_{i=1}^m \varepsilon_i \frac{1}{w_{i1}w_{i2}}+\sum_{j=1}^k (-n_j).$
\end{center}
\item \begin{center}
$\displaystyle 3L(M)=\sum_{i=1}^m \varepsilon_i \frac{w_{i1}^2+w_{i2}^2}{w_{i1}w_{i2}}+\sum_{j=1}^k n_j.$
\end{center}
\end{enumerate}
\end{theorem}

\begin{proof}
We consider Equation \eqref{ASformula}. Let $t=z-1$. First we show that 
$$
\frac{(1+z^a)(1+z^b)}{(1-z^a)(1-z^b)}=\frac{4}{ab}t^{-2} + \frac{4}{ab}t^{-1} +\frac{a^2+b^2+1}{3ab} +O(t),
$$
where $a$ and $b$ are positive integers. Indeed,
\begin{center}
$\displaystyle \frac{(1+z^a)(1+z^b)}{(1-z^a)(1-z^b)}=\frac{(1+(1+t)^a)(1+(1+t)^b)}{(1-(1+t)^a)(1-(1+t)^b)}=\frac{(2+at+\frac{a(a-1)}{2}t^2+O(t^3))(2+bt+\frac{b(b-1)}{2}t^2+O(t^3))}{abt^2(1+\frac{a-1}{2}t+\frac{(a-1)(a-2)}{6}t^2+O(t^3))(1+\frac{b-1}{2}t+\frac{(b-1)(b-2)}{6}t^2+O(t^3))}.$
\end{center}
Letting $-A=\frac{a-1}{2}t+\frac{(a-1)(a-2)}{6}t^2+O(t^3)$ and using the geometric series expansion $\frac{1}{1-A}=1+A+A^2+\cdots$,
\begin{center}
$\displaystyle \frac{2+at+\frac{a(a-1)}{2}t^2+O(t^3)}{1+\frac{a-1}{2}t+\frac{(a-1)(a-2)}{6}t^2+O(t^3)}$

$\displaystyle =\left(2+at+\frac{a(a-1)}{2}t^2+O(t^3)\right)(1+A+A^2+\cdots)$

$\displaystyle =2-(a-1)t+at-\frac{(a-1)(a-2)}{3}t^2+\frac{(a-1)^2}{2}t^2-\frac{a(a-1)}{2}t^2+\frac{a(a-1)}{2}t^2+O(t^3)$

$\displaystyle =2+t+\frac{a^2-1}{6}t^2+O(t^3).$
\end{center}
Therefore, 
\begin{center}
$\displaystyle \frac{(1+z^a)(1+z^b)}{(1-z^a)(1-z^b)}$

$\displaystyle =\frac{1}{abt^2}\left(2+t+\frac{a^2-1}{6}t^2+O(t^3) \right) \left(2+t+\frac{b^2-1}{6}t^2+O(t^3)\right)$

$\displaystyle =\frac{1}{abt^2}\left(4+4t+\frac{1}{3}(a^2+b^2+1)t^2+O(t^3)\right).$
\end{center}
Moreover,
$$\frac{z}{(1-z)^2}=\frac{1}{t^2}+\frac{1}{t}.$$

Therefore, comparing the constant terms in Equation \eqref{ASformula} proves the first equation. Moreover, comparing the coefficients of the $t^{-2}$-terms in Equation \eqref{ASformula}, we get
$$
0=\sum_{i=1}^m \varepsilon_i \frac{4}{w_{i1}w_{i2}}+\sum_{j=1}^k (-4n_j),
$$
which proves the second equation. The first and second equations imply the third equation.
\end{proof}

If there are only isolated fixed points, Theorem \ref{3formulas} implies the following formulas.

\begin{cor}
Let $S^1$ act on a 4-dimensional compact oriented manifold $M$. Suppose that there are only isolated fixed points. The following formulas hold.
\begin{enumerate}
\item \begin{center}
$\displaystyle L(M)=\sum_{i=1}^m \varepsilon_i \frac{w_{i1}^2+w_{i2}^2}{3w_{i1}w_{i2}}.$
\end{center}
\item \begin{center}
$\displaystyle 0=\sum_{i=1}^m \varepsilon_i \frac{1}{w_{i1}w_{i2}}.$
\end{center}
\end{enumerate}
\end{cor}

\begin{proof}
Since there are only isolated fixed points, (2) of Theorem \ref{3formulas} implies the second formula of this corollary. Together with (1) of Theorem \ref{3formulas}, this implies the first formula.
\end{proof}

Let $M$ be a $2n$-dimensional compact oriented $S^1$-manifold. For an isolated fixed point $p_i$, denote by $k_i$ the number of weights at $p_i$ that are equal to 1. If the fixed point set of $M$ consists of isolated points only, \cite[Lemma 1.1]{mus80} states that 
$$
\sum_i {\varepsilon_i\,k_i}=0;
$$
also see \cite{jan18}. Now let $M$ have dimension 4, and suppose $M$ has isolated fixed points $p_1$, $\cdots$, $p_m$ and fixed surfaces $F_1$, $\cdots$, $F_k$, Then $k_i$ can be 0, 1, or 2. The below lemma, which is an extension in dimension 4 of the above fact, can be derived from Equation \eqref{ASformula} by the same way as \cite[Lemma 1.1]{mus80}, as we prove below.

\begin{lemma}\label{weight12}
Let $M$ be a 4-dimensional compact oriented $S^1$-manifold with isolated fixed points $p_1,\cdots,p_m$ and fixed surfaces $F_1,\cdots,F_k$. Then
\begin{equation} \label{weight1'}
\sum_{i=1}^m {\varepsilon_i\,k_i}-2\sum_{j=1}^k{n_j}=0,
\end{equation}
where $k_i$ is the multiplicity of weight 1 at $p_i$.
\end{lemma}

\begin{proof}
By Theorem \ref{AHformula}, the signature of $M$ satisfies
$$
\mathrm{sign}(M) =\sum\limits_{i=1}^m\varepsilon_i\, {\frac{(z^{w_{i1}}+1)(z^{w_{i2}}+1)}{(z^{w_{i1}}-1)(z^{w_{i2}}-1)}}-\sum\limits_{j=1}^k\frac{4zn_j}{(1-z)^2}.
$$
This formula holds for all indeterminate $z$, and is a constant.

For an isolated fixed point $p_i$,
\begin{center}
$\displaystyle \left. \frac{d}{dz}  \left( \frac{(z^{w_{i1}}+1)(z^{w_{i2}}+1)}{(z^{w_{i1}}-1)(z^{w_{i2}}-1)} \right)\right|_{z=0}$

$\displaystyle =
\begin{cases} 
0 & \mbox{if none of } w_{i1} \mbox{ and } w_{i2} \mbox{ are equal to 1} \\
2 & \mbox{if exactly one of } w_{i1} \mbox{ and } w_{i2} \mbox{ is equal to 1} \\
4 & \mbox{if both } w_{i1} \mbox{ and } w_{i2} \mbox{ are equal to 1}
\end{cases}$
\end{center}

For a fixed surface $F_j$,
$$ \left. \frac{d}{dz}  \left(  \frac{4zn_j}{(1-z)^2}  \right) \right|_{z=0}=4n_j.$$

Therefore, taking the derivative of the above formula and evaluating at $z=0$, the lemma follows.
\end{proof}

Lemma \ref{weight12} enables us to associate a graph to a 4-dimensional compact oriented $S^1$-manifold, as we do so in the next section.

\section{Graph} \label{s4}

\subsection{Virtual graph of weights}

To a 4-dimensional compact oriented $S^1$-manifold, we shall associate a graph that contains information on data on the fixed point set. For this, we introduce terminologies.

\begin{defn} \label{defngraph}
A (4-dimensional virtual) \textbf{graph of weights} is a graph defined as follows.
\begin{enumerate}
\item Its vertex set consists of two types of vertices, called \textbf{points} $p_1$, $\cdots$, $p_m$ and \textbf{surfaces} $F_1$, $\cdots$, $F_k$.
\item To each point $p_i$ is associated sign $+1$ or $-1$, denoted $\varepsilon_i$.
\item To each surface $F_j$ is associated an integer $n_j$.
\item Each point $p$ has two edges, and the edges are labeled by positive integers that are relatively prime, called \textbf{weights} of $p$.
\item Each surface has $2|n_j|$ edges, each of which has label 1.
\end{enumerate}
\end{defn}

\begin{defn} \label{graphofweights}
Let $M$ be a 4-dimensional compact oriented $S^1$-manifold. We say that $G_W$ is a \textbf{graph of weights of} $M$ if the following hold.
\begin{enumerate}
\item The vertices of $G_W$ are the fixed components of $M$.
\item The sign of a point in $G_W$ is the sign of the corresponding isolated fixed point.
\item The labels of edges of a point of $G_W$ are the weights of the corresponding isolated fixed point.
\item For a surface $F_j$ in $G_W$, the number $n_j$ is the Euler number of the normal bundle of the corresponding fixed surface.
\item For a fixed surface $F_j$ in $M$, the corresponding surface of $G_W$ has $2|n_j|$ edges, all of label 1.
\item If two vertices are connected by an edge whose label $w$ is bigger than 1, then the corresponding fixed points are in the same component of $M^{\mathbb{Z}_w}$, which is an invariant 2-sphere with weight $w$. 
\end{enumerate}
\end{defn}

With the terminologies, we show that for a 4-dimensional compact oriented $S^1$-manifold, we can encode its fixed point data in a graph.

\begin{theorem} \label{graph}
Let $M$ be a 4-dimensional compact oriented $S^1$-manifold. Then a graph $G_W$ of weights for $M$ exists. Moreover, there is one that has no self-loops with the following property: if there is an edge of label 1 between two vertices, then the corresponding fixed components contribute the terms with different signs in Equation \eqref{weight1'} of Lemma \ref{weight12}.
\end{theorem}

\begin{proof}
Let $p_1$, $\cdots$, $p_m$ denote isolated fixed points and let $F_1$, $\cdots$, $F_k$ be fixed surfaces. To each $p_i$ we assign a vertex, also denoted by $p_i$, which is a point; the sign of the vertex $p_i$ is the sign $\varepsilon_i$ of the corresponding fixed point $p_i$. To each $F_j$ we assign a vertex, also denoted $F_j$, which is a surface.

Suppose that a fixed point $p_i$ has weight $w$ that is bigger than 1. 
By Lemma \ref{isotropysphere}, $p_i$ lies in a 2-sphere, which is a component of $M^{\mathbb{Z}_w}$ and contains another fixed point $p_l$ that has a weight $w$.
We draw an edge between $p_i$ and $p_l$ and give the edge a label $w$. 

Next, we draw edges for weight 1. We rewrite Equation \eqref{weight1'} as
\begin{equation} \label{weight1re}
\displaystyle \sum_{\varepsilon_i >0}  k_i - 2 \sum_{n_j <0} n_j=\sum_{\varepsilon_i <0}  k_i + 2 \sum_{n_j >0} n_j,
\end{equation}
so that each fixed component contributes Equation \eqref{weight1re} in one side of the equation with a positive coefficient. If $\varepsilon_i>0$ ($\varepsilon_i<0$), then $p_i$ contributes to Equation \eqref{weight1re} in the left side (right side) by $k_i$, which is positive. Similarly, if $n_j<0$ ($n_j>0$) then $F_j$ contributes to Equation \eqref{weight1re} in the left side (right side) by $-2n_j$ ($2n_j$), which is positive. Now, to each isolated fixed point $p_i$ we assign $k_i$ edges of label 1 and to each fixed surface $F_j$, we assign $2n_j$ edges of label 1, so that an edge of label 1 has vertices (fixed components) $v$ and $v'$ whose contributions as fixed components in Equation \eqref{weight1re} are in the opposite sides.
\end{proof}

We illustrate a graph of weights with an example.

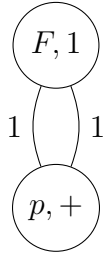
\begin{figure}
\centering
\begin{tikzpicture}[state/.style ={circle, draw}]
\node[state] (a) {$p,+$};
\node[state] (b) [above=of a] {$F,1$};
\path (a) [bend right =20]edge node[right] {$1$} (b);
\path (b) [bend right =20]edge node[left] {$1$} (a);
\end{tikzpicture}
\caption{Graph of weights for Example \ref{P(1,0)}}\label{figCP}
\end{figure}

\begin{exa} \label{P(1,0)}
Let $S^1$ act on the complex projective space $\mathbb{CP}^2$ by
$$g \cdot [z_0:z_1:z_2]=[g z_0:z_1:z_2]$$
for all $g \in S^1 \subset \mathbb{C}$ and $[z_0:z_1:z_2] \in \mathbb{CP}^2$. There are two fixed components, $p=[1:0:0]$ and $F=[0:z_1:z_2]$. The complex weights at $p$ are $-1,-1$ and thus $\varepsilon_p=+1$ and $w_{p1}=w_{p2}=1$. 

The Euler number of the normal bundle of $F$ is 1. Then Figure \ref{figCP} is a graph of weights for $\mathbb{CP}^2$ with the action. The fixed point $p$ has weights $\{1,1\}$ and sign $+1$ and so a corresponding vertex has two edges of label 1 and has sign $+$. The fixed surface $F$ has Euler number 1 for its normal bundle, and so its corresponding vertex is assigned a number 1 and has 2 edges of label 1. In this case, the Atiyah-Hirzebruch formula (Theorem \ref{AHformula}) is
$$
\mathrm{sign}(\mathbb{CP}^2)={\frac{(z+1)(z+1)}{(z-1)(z-1)}}-\frac{4z}{(z-1)^2}=\frac{z^2-2z+1}{(z-1)^2}=1. 
$$

There are two 2-spheres $F_1=[z_0:z_1:0]$ and $F_2=[z_0:0:z_2]$ connecting the fixed components $p$ and $F$, which are both invariant 2-spheres with weight 1.
\end{exa}

\begin{exa} \label{S(1,0)}
Let $S^4=\{(z_1,z_2,x) \in \mathbb{C} \times \mathbb{C} \times \mathbb{R} : |z_1|^2+|z_2|^2+x^2=1\}$ be the 4-sphere. Let $S^1$ act on the 4-sphere $S^4$ by
\begin{center}
$g \cdot (z_1,z_2,x)=(g z_1, z_2,x)$
\end{center}
for all $g \in S^1 \subset \mathbb{C}$ and for all $(z_1,z_2,x) \in S^4$. The action has one fixed component $F:=\{(0,z_2,x) \in S^4\}$, which is a 2-sphere. The Euler number of the normal bundle of $F$ is 0. Thus, a graph of weights for this action on $S^4$ consists of one vertex that is a surface, with 0 assigned to the vertex.
\end{exa}

\begin{remark} \label{non-uniqueness}
A graph of weights does not determine an $S^1$-manifold uniquely. Here is an example. Consider the $S^1$-action on $\mathbb{CP}^2$ in Example \ref{P(1,0)} and its graph of weights Figure \ref{figCP}. Let $N$ be any 4-dimensional compact connected oriented $S^1$-manifold that has no fixed points. Take an equivariant connected sum along free orbits of Example \ref{P(1,0)} and $N$; let $M$ denote the connected sum. Then Figure \ref{figCP} is also a graph of weights for $M$, but Example \ref{P(1,0)} and $M$ are in general not (equivariantly) diffeomorphic.
\end{remark}

\subsection{Euler number of an edge}

Given a graph of weights, for an edge we shall define a number and call it the Euler number of the edge. If the graph is of a 4-dimensional compact oriented $S^1$-manifold, for an edge of label bigger than 1, we shall see that the Euler number of an edge is exactly the Euler number of the normal bundle of the invariant 2-sphere corresponding to the edge.

\begin{defn} \label{OtherWeight} Let $G_W$ be a graph of weights. Let $e$ be an edge with label $w_e$, and let $v$ be a vertex of $e$. We define a number, called the \textbf{other weight} of $w_e$ at $v$, as follows.
\begin{enumerate}
\item If $v$ is a point, then the other weight of $w_e$ at $v$ is the label of the other edge of $v$.
\item If $v$ is a surface, then the other weight of $v$ is 0.
\end{enumerate}
\end{defn}

\begin{defn} \label{EulNum}
Let $G_W$ be a graph of weights. Let $e$ be an edge with label $w_e$. Let $v$ and $v'$ be the vertices of $e$, and let $w$ and $w'$ be the other weight of $w_e$ at $v$ and $v'$, respectively. We define a rational number, denoted $n_e$ and called the \textbf{Euler number of $e$}, by
\begin{center}
$\displaystyle n_e:=\frac{\varepsilon_v w+\varepsilon_{v'} w'}{w_e}.$
\end{center}
Here, if $v$ is a surface, we define $\varepsilon_v$ to be 1 and similarly for $v'$.
\end{defn}

\begin{defn}
We say that a graph of weights is \textbf{admissible} if the Euler number of each edge is an integer.
\end{defn}

For an integer $n$, we consider a complex line bundle $\mathcal{O}(n)$ over $\mathbb{CP}^1$ with Euler number $n$, the quotient of $(\mathbb{C}^2-\{0\}) \times \mathbb{C}$ by a $C^*$-action

$$z \cdot (z_1,z_2,w)=(zz_1,zz_2,z^n w).$$

Let $S^1$ act on $\mathcal{O}(n)$ by

$$g \cdot [z_1,z_2,w]=[z_1,g^{u_1}z_2,g^{u_2}w]$$
for some integers $u_1$ and $u_2$. 

Suppose that $u_1 \neq 0$. Then this action has two fixed points $q_1=[1,0,0]$ and $q_2=[0,1,0]$ that have (complex) weights $\{u_1,u_2\}$ and $\{-u_1,-nu_1+u_2\}$. If we let $u_1=w_e$, $u_2=\epsilon_i w_{i2}$, $-nu_1+u_2=-\epsilon_k w_{k2}$, then we have
$$n=\frac{\epsilon_i w_{i2}+\epsilon_k w_{k2}}{w_e}.$$

\begin{theorem} \label{n_e}
Let $M$ be a 4-dimensional compact oriented $S^1$-manifold. Let $G_W$ be its graph of weights. Let $e$ be an edge whose weight is bigger than 1. Then the Euler number $n_e$ of the edge is equal to the Euler number of the normal bundle of the invariant 2-sphere of weight $w_e$, containing two fixed points that correspond to the vertices of $e$.
\end{theorem}

\begin{proof}
Let $p_i$ and $p_k$ be the vertices (fixed points) of $e$ and let $w=w_{i1}=w_{k1}$ be the weight of the edge $e$. By (1) of Theorem \ref{graph}, $p_i$ and $p_k$ are in an invariant 2-sphere $S_e$ of weight $w$.
Choose an orientation of $S_e$, so that $S^1$ acts on $T_{p_i}S_e$ with weight $w_{i1}$ and on $T_{p_k}S_e$ with weight $-w_{k1}$. Then $S^1$ acts on $N_{p_i}S_e$ with weight $\epsilon_i w_{i2}$ and on $N_{p_k}S_e$ with weight $-\epsilon_k w_{k2}$.

The normal bundle of $S_e$ is orientation preserving equivariantly diffeomorphic to $\mathcal{O}(n)$ for some integer $n$ with $u_1=w_{i1}$ and $u_2=\epsilon_i w_{i2}$. Thus, $-\epsilon_k w_{k2}$ is equal to $-n u_1+u_2=-n w_{i1}+\epsilon_i w_{i2}$, that is, $-\epsilon_k w_{k2}=-n w_{i1}+\epsilon_i w_{i2}$. Therefore,  the Euler number $n$ of the normal bundle of $S_e$ is $n=\frac{\varepsilon_i w_{i2}+\varepsilon_k w_{k2}}{w_{i1}}$, which is the Euler number $n_e$ of the edge $e$. \end{proof}

Any graph of weights of a 4-dimensional compact oriented $S^1$-manifold is admissible.

\begin{prop} \label{admi}
Let $M$ be a 4-dimensional compact oriented $S^1$-manifold and let $G_W$ be its graph of weights. Then $G_W$ is admissible.
\end{prop}

\begin{proof}
Let $e$ be an edge of $G_W$. Let $v$ and $v'$ be the vertices of $e$, and let $w$ ($w'$) be the other weight of $v$ ($v'$).

If its label $w_e$ is 1, then its Euler number $n_e=\frac{\varepsilon_v w+\varepsilon_{v'} w'}{w_e}=\varepsilon_v w+\varepsilon_{v'} w'$ is an integer. 

If its label $w_e$ is bigger than 1, then Theorem \ref{n_e} implies that its Euler number $n_e$ is an integer. 
\end{proof}

With the above theorem, we give another description of the $L$-genus of a 4-dimensional oriented $S^1$-manifold in terms of $n_e$ and $n_j$.

\begin{theorem}
Let $M$ be a 4-dimensional compact oriented $S^1$-manifold and $G_W$ its graph of weights. Then
$$
3 \, L(M)=\sum\limits_{e\in E(G_W)} {n_e}+\sum\limits_{j=1}^k{n_j},
$$
where $E(G_W)$ denotes the set of edges of $G_W$.
\end{theorem}

\begin{proof}
By the last equation of Theorem \ref{3formulas},
$$
3L(M)=\sum_{i=1}^m \varepsilon_i \left(\frac{w_{i1}}{w_{i2}}+\frac{w_{i2}}{w_{i1}} \right)+\sum_{j=1}^k n_j.
$$

Let $e$ be an edge of $G_W$. Let $w_e$ denote its label.

First, suppose that the edge is between two isolated fixed points $p_i$ and $p_k$. Let $w$ and $w'$ be the other weight at $p_i$ and $p_k$, respectively. Then $\varepsilon_i \frac{w}{w_e}$ is one of the two terms $\varepsilon_i(\frac{w_{i1}}{w_{i2}}+\frac{w_{i2}}{w_{i1}})$ and similarly $\varepsilon_k \frac{w'}{w_e}$ is one of the two terms $\varepsilon_k(\frac{w_{k1}}{w_{k2}}+\frac{w_{k2}}{w_{k1}})$ in (3) of Theorem \ref{3formulas}. By Definition \ref{EulNum}, their sum $\frac{\varepsilon_v w+\varepsilon_{v'} w'}{w_e}$ is the Euler number $n_e$ of the edge $e$.

Second, suppose that the edge is between an isolated fixed point $p_i$ and a fixed surface $F_k$. Let $w$ be the other weight at $p_i$. The term $\varepsilon_i \frac{w}{w_e}$ is one of the two terms $\varepsilon_i(\frac{w_{i1}}{w_{i2}}+\frac{w_{i2}}{w_{i1}})$ in (3) of Theorem \ref{3formulas}. The weight of the tangent bundle of $F_k$ is zero, and this agrees with the other weight $w'$ of $F_k$ as a vertex of $G_W$. By Definition \ref{EulNum}, the sum $\frac{\varepsilon_v w+\varepsilon_{v'} w'}{w_e}=\frac{\varepsilon_v w}{w_e}$ is the Euler number $n_e$ of the edge $e$.

Third, suppose that the edge is between two fixed surfaces $F_i$ and $F_k$. The weight of the tangent bundle of $F_i$ ($F_k$) is zero, which is the other weight $w$ ($w'$) of the vertex $F_i$ ($F_k$, respectively). Thus $0=\frac{\varepsilon_v w+\varepsilon_{v'} w'}{w_e}$ is the Euler number $n_e$ of the edge $e$.

To sum up, by rearranging the terms $\varepsilon_i (\frac{w_{i1}}{w_{i2}}+\frac{w_{i2}}{w_{i1}} )$ for the edges, we get
$$
\sum_{i=1}^m \varepsilon_i \left(\frac{w_{i1}}{w_{i2}}+\frac{w_{i2}}{w_{i1}} \right)=\sum\limits_{e\in E(G)} \frac{\varepsilon_iw_{i2}+\varepsilon_kw_{k2}}{w_e}=\sum\limits_{e\in E(G)} {n_e}.
$$
This proves the theorem. \end{proof}

Let $M$ be a 4-dimensional compact oriented manifold and let $G_W$ be its graph of weights. By Theorem \ref{n_e}, every edge with label bigger than 1 corresponds to an invariant 2-sphere. It is a natural question to ask, if there exists a graph of weights, in which an edge with label 1 also corresponds to an invariant 2-sphere.

\begin{que} \label{quesphere}
Let $M$ be a 4-dimensional compact oriented $S^1$-manifold. Does there exists a graph $G_W$ of weights, such that any edge $e$ of label $1$ between vertices $p$ and $q$ corresponds to an invariant 2-sphere with weight $1$ between fixed components $p$ and $q$?
\end{que}

For instance, the answer to Question \ref{quesphere} is affirmative in Example \ref{P(1,0)}.

\subsection{Graphs for basic manifolds}

We illustrate graphs of weights for basic 4-dimensional oriented $S^1$-manifolds.

\begin{figure}
\centering
\begin{subfigure}[b][5.5cm][s]{.17\textwidth}
\centering
\vfill
\begin{tikzpicture}[state/.style ={circle, draw}]
\node[state] (a) {$p_1,+$};
\node[state] (b) [above=of a] {$p_2,-$};
\path (a) [bend right =20]edge node[right] {$b$} (b);
\path (b) [bend right =20]edge node[left] {$a$} (a);
\end{tikzpicture}
\vfill
\caption{$S(a,b)$}\label{fig1-1}
\vspace{\baselineskip}
\end{subfigure}\qquad
\begin{subfigure}[b][5.5cm][s]{.3\textwidth}
\centering
\vfill
\begin{tikzpicture}[state/.style ={circle, draw}]
\node[state] (a) {$p_1,+$};
\node[state] (b) [above right=of a] {$p_2,-$};
\node[state] (c) [above left=of b] {$p_3,+$};
\path (a) edge node[right] {$a$} (b);
\path (b) edge node [right] {$b-a$} (c);
\path (a) edge node [left] {$b$} (c);
\end{tikzpicture}
\vfill
\caption{$P(a,b)$}\label{fig1-2}
\vspace{\baselineskip}
\end{subfigure}\qquad
\begin{subfigure}[b][5.5cm][s]{.3\textwidth}
\centering
\vfill
\begin{tikzpicture}[state/.style ={circle, draw}]
\node[state] (a) {$p_1,-$};
\node[state] (b) [above right=of a] {$p_2,+$};
\node[state] (c) [above left=of b] {$p_3,-$};
\path (a) edge node[right] {$a$} (b);
\path (b) edge node [right] {$b-a$} (c);
\path (a) edge node [left] {$b$} (c);
\end{tikzpicture}
\vfill
\caption{$Q(a,b)$}\label{fig1-3}
\vspace{\baselineskip}
\end{subfigure}\qquad
\caption{Graphs for $S(a,b)$, $P(a,b)$, $Q(a,b)$}\label{fig1}
\end{figure}
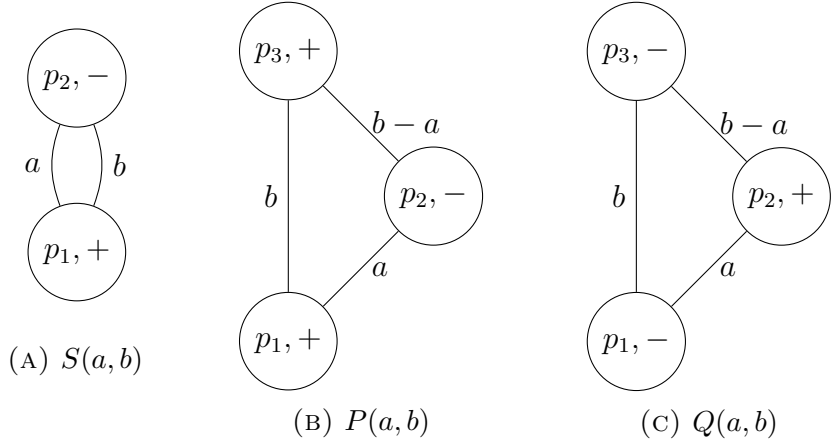

\begin{exa} \label{S4}
Let $a$ and $b$ be positive integers. Let $S(a,b)$ denote an $S^1$-action on $S^4$  with weights $a$ and $b$. That is, let $S^1$ act on the 4-sphere $S^4=\{(z_1,z_2,x) \in \mathbb{C}^2 \times \mathbb{R} : |z_1|^2+|z_2|^2+x^2=1\}$ by
\begin{center}
$g \cdot (z_1,z_2,x)=(g^a z_1, g^b z_2, x)$.
\end{center}
The action has two fixed points $p_1=(0,0,1)$ and $p_2=(0,0,-1)$. Both fixed points have weights $\{a,b\}$, and the sign of $p_1$ is $+1$ and the sign of $p_2$ is $-1$. Figure \ref{fig1-1} is the graph of weights for $S^4$ with this action.
\end{exa}

The above example and its graph can be extended to any even sphere; see \cite{jan18, jan22}. Any $S^1$-action on a compact oriented manifold with two fixed points have the same signs and weights at the fixed points as a linear $S^1$-action on $S^{2n}$ \cite{jan18, jan20, kos84}. In addition, see \cite{mus16} for $S^1$-actions with two fixed points on unitary manifolds.

\begin{exa}
Let $0<a<b$ be integers. Consider an action of ${S}^1$ on the complex projective space $\mathbb{CP}^2$ by
$$
g \cdot [z_0:z_1:z_2]=[z_0:g^a z_1:g^b z_2]
$$
for all $g \in S^1 \subset \mathbb{C}$. The action has 3 fixed points, $p_1=[1:0:0]$, $p_2=[0:1:0]$, and $p_3=[0:0:1]$, whose signs are $+1$, $-1$, and $+1$, and weights are $\{a,b\}$, $\{a, b-a\}$, and $\{b-a, b\}$, as their complex weights are $\{a,b\}$, $\{-a,b-a\}$, and $\{-b,a-b\}$, respectively. Denote $\mathbb{CP}^2$ with this action by $P(a,b)$ and $-\mathbb{CP}^2$ by $Q(a,b)$, where $-\mathbb{CP}^2$ denotes $\mathbb{CP}^2$ with the opposite orientation. Figure \ref{fig1-2} is a graph of weights for $P(a,b)$, and Figure \ref{fig1-3} is a graph of weights for $Q(a,b)$.
\end{exa}

\section{Equivariant connected sum and splitting} \label{s5}

Connected sum and splitting are classical operations in topology. If two manifolds of same dimension admit torus actions, then we may take a connected sum equivariantly, so that the connected sum is equipped with a torus action and similarly for splitting. Orlik and Raymond considered equivariant splitting of a $T^2$-action on a compact, simply connected oriented 4-manifold into a rather minimal manifold \cite{OR1} and Fintushel considered that of an $S^1$-action \cite{F2}. They showed that such a manifold is an equivariant connected sum of a homology $S^4$ and copies of $\pm \mathbb{CP}^2$ and $S^2 \times S^2$. 
Via equivariant splitting, the second author studied generators of unitary $S^1$-bordism ring \cite{mus83}. In this section, we discuss equivariant connected sum and splitting, of a circle action on an oriented 4-manifold.

\subsection{Connected sum and splitting} \label{s5.1}

Let $M$ be a 4-dimensional compact oriented $S^1$-manifold. Let $p$ be an isolated fixed point. Let $a$ and $b$ be the weights at $p$.

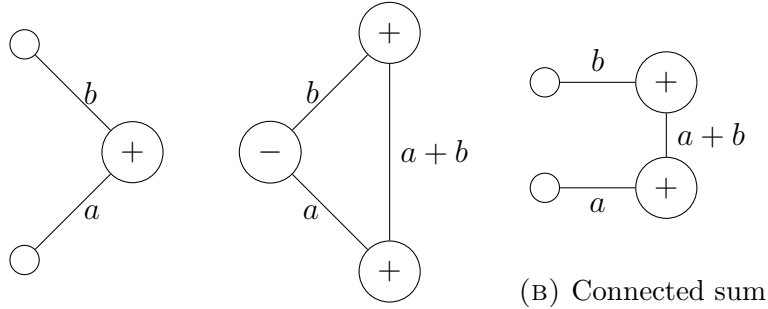
\begin{figure}
\centering
\begin{subfigure}[b][4.6cm][s]{.4\textwidth}
\centering
\vfill
\begin{tikzpicture}[state/.style ={circle, draw}]
\node[state] (a) {$+$};
\node[state] (b) [below left=of a] {};
\node[state] (c) [above left=of a] {};
\node[state] (d) [right=of a]{$-$};
\node[state] (e) [below right=of d] {$+$};
\node[state] (f) [above right=of d] {$+$};
\path (a) edge node[right] {$a$} (b);
\path (a) edge node [right] {$b$} (c);
\path (d) edge node[left] {$a$} (e);
\path (e) edge node [right] {$a+b$} (f);
\path (d) edge node [left] {$b$} (f);
\end{tikzpicture}
\vfill
\caption{Graphs of weights for $M$ and $P(a,a+b)$}\label{fig4-1}
\vspace{\baselineskip}
\end{subfigure}\qquad
\begin{subfigure}[b][4.6cm][s]{.4\textwidth}
\centering
\vfill
\begin{tikzpicture}[state/.style ={circle, draw}]
\node[state] (a) {};
\node[state] (b) [above=of a] {};
\node[state] (c) [right=of a] {$+$};
\node[state] (d) [right=of b] {$+$};
\path (a) edge node[below] {$a$} (c);
\path (b) edge node[above] {$b$} (d);
\path (c) edge node[right] {$a+b$} (d);
\end{tikzpicture}
\vfill
\caption{Connected sum}\label{fig4-2}
\vspace{\baselineskip}
\end{subfigure}\qquad
\caption{Equivariant connected sum of $M$ and $P(a,a+b)$ at $p$ and $[0:1:0]$; its converse is equivariant splitting}\label{fig4}
\end{figure}

Suppose that the sign of $p$ is $+1$. Consider the manifold $P(a,a+b)$; its fixed point $[0:1:0]$ has weights $\{a,b\}$ and sign $-1$. Thus, there is an orientation reversing $S^1$-equivariant diffeomorphism between a neighborhood of $p$ in $M$ and a neighborhood of $[0:1:0]$ in $P(a,a+b)$. By using the equivariant diffeomorphism, we can take an equivariant connected sum of $M$ and $P(a,a+b)$ at $p$ and $[0:1:0]$, to get another 4-dimensional compact oriented $S^1$-manifold $N$. Figure \ref{fig4-1} is a graph of weights for $M \sqcup P(a,a+b)$, and Figure \ref{fig4-2} is one for $N$.

The converse of the above equivariant connected sum is equivariant splitting. Let $N$ be a 4-dimensional compact oriented $S^1$-manifold and suppose that there is an invariant 2-sphere with weight $a+b$ between two fixed points $q$ and $q'$, which have weights $\{a,a+b\}$ and $\{b,a+b\}$ and both have sign $+1$. Then by the reverse operation of the above equivariant connected sum, we can equivariantly split $N$ into another 4-dimensional compact oriented $S^1$-manifold $M$ and $P(a,a+b)$, whose graphs of weights are the left and the right of Figure \ref{fig4-1}, respectively.

\begin{figure}
\centering
\begin{subfigure}[b][4.6cm][s]{.4\textwidth}
\centering
\vfill
\begin{tikzpicture}[state/.style ={circle, draw}]
\node[state] (a) {$-$};
\node[state] (b) [below left=of a] {};
\node[state] (c) [above left=of a] {};
\node[state] (d) [right=of a]{$+$};
\node[state] (e) [below right=of d] {$-$};
\node[state] (f) [above right=of d] {$-$};
\path (a) edge node[right] {$a$} (b);
\path (a) edge node [right] {$b$} (c);
\path (d) edge node[left] {$a$} (e);
\path (e) edge node [right] {$a+b$} (f);
\path (d) edge node [left] {$b$} (f);
\end{tikzpicture}
\vfill
\caption{Graphs of weights for $M$ and $Q(a,a+b)$}\label{fig5a}
\vspace{\baselineskip}
\end{subfigure}\qquad
\begin{subfigure}[b][4.6cm][s]{.4\textwidth}
\centering
\vfill
\begin{tikzpicture}[state/.style ={circle, draw}]
\node[state] (a) {};
\node[state] (b) [above=of a] {};
\node[state] (c) [right=of a] {$-$};
\node[state] (d) [right=of b] {$-$};
\path (a) edge node[below] {$a$} (c);
\path (b) edge node[above] {$b$} (d);
\path (c) edge node[right] {$a+b$} (d);
\end{tikzpicture}
\vfill
\caption{Connected sum}\label{fig5b}
\vspace{\baselineskip}
\end{subfigure}\qquad
\caption{Equivariant connected sum of $M$ and $Q(a,a+b)$ at $p$ and $[0:1:0]$; its converse is equivariant splitting}\label{fig5}
\end{figure}
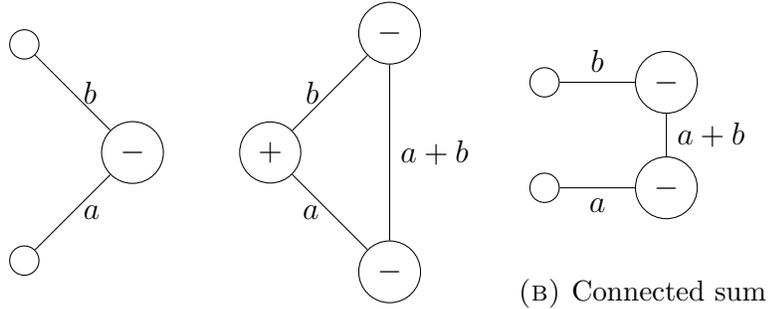

Suppose that the sign of $p$ is $-1$. In this case, we can take an equivariant connected sum $N$ of $M$ and $Q(a,a+b)$ at $p$ and $[0:1:0]$; now $[0:1:0]$ has the same weights $\{a,b\}$ but has sign $+1$. Figure \ref{fig5a} is a graph of weights for $M \sqcup Q(a,a+b)$, and Figure \ref{fig5b} is one for $N$. Conversely, if in a 4-dimensional compact oriented $S^1$-manifold $N$ there is an invariant 2-sphere with weight $a+b$ between two fixed points that have weights $\{a,a+b\}$ and $\{b,a+b\}$ and both fixed points have sign $-1$, we can equivariant split $N$ into $M$ and $Q(a,a+b)$; Figure \ref{fig5b} is a graph of weights for $N$ and Figure \ref{fig5a} is one for $M \sqcup Q(a,a+b)$.

\begin{figure}
\centering
\begin{subfigure}[b][5cm][s]{.5\textwidth}
\centering
\vfill
\begin{tikzpicture}[state/.style ={circle, draw}]
\node[state] (a) {$p,+$};
\node[state] (b) [below left=of a] {$+$};
\node[state] (c) [above left=of a] {$-$};
\node[state] (d) [right=of a]{$q,-$};
\node[state] (e) [below right=of d] {$-$};
\node[state] (f) [above right=of d] {$+$};
\path (a) edge node[right] {$a$} (b);
\path (a) edge node [right] {$b$} (c);
\path (b) edge node [right] {$b-a$} (c);
\path (d) edge node[left] {$a$} (e);
\path (e) edge node [right] {$b-a$} (f);
\path (d) edge node [left] {$b$} (f);
\end{tikzpicture}
\vfill
\caption{Graphs of weights for $P(a,b)$ and $Q(a,b)$}\label{fig5a-}
\vspace{\baselineskip}
\end{subfigure}\qquad
\begin{subfigure}[b][5cm][s]{.4\textwidth}
\centering
\vfill
\begin{tikzpicture}[state/.style ={circle, draw}]
\node[state] (a) {$p_1,+$};
\node[state] (b) [above=of a] {$p_4,-$};
\node[state] (c) [right=of a] {$p_2,-$};
\node[state] (d) [right=of b] {$p_3,+$};
\path (a) edge node[right] {$b-a$} (b);
\path (a) edge node[below] {$a$} (c);
\path (b) edge node[above] {$b$} (d);
\path (c) edge node[right] {$b-a$} (d);
\end{tikzpicture}
\vfill
\caption{Connected sum of $P(a,b)$ and $Q(a,b)$ at $[1:0:0]$}\label{fig5b-}
\vspace{\baselineskip}
\end{subfigure}\qquad
\caption{Equivariant connected sum of $P(a,b)$ and $Q(a,b)$ at $[1:0:0]$}\label{fig5-}
\end{figure}
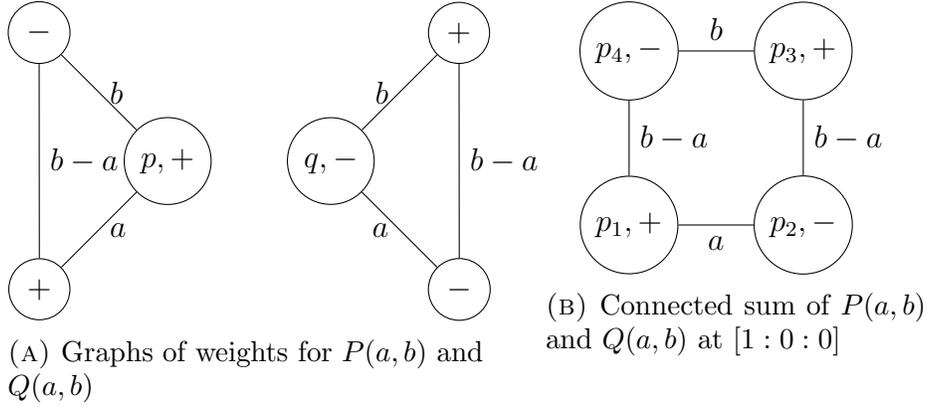

\begin{exa} \label{e23}
Consider the two $S^1$-manifolds $P(a,b)$ and $Q(a,b)$. Their graphs of weights are the left and right figures of Figure \ref{fig5-}; note that the positions of vertices are different from Figures \ref{fig1-2} and \ref{fig1-3}. The fixed point $p=[1:0:0]$ in $P(a,b)$ has sign $+1$ and weights $\{a,b\}$ and the fixed point $q=[1:0:0]$ in $Q(a,b)$ has sign $-1$ and weights $\{a,b\}$. Therefore, Let $P\# Q(a,b)$ denote a connected sum of $P(a,b)$ and $Q(a,b)$ at $p$ and $q$. Then Figure \ref{fig5b-} is a graph of weights of $P\# Q(a,b)$.
\end{exa}

\begin{figure}
\centering
\begin{subfigure}[b][4.3cm][s]{.8\textwidth}
\centering
\vfill
\begin{tikzpicture}[state/.style ={circle, draw}]
\node[state] (a) {$q,-$};
\node[state] (b) [left=of a] {};
\node[state] (c) [above=of a] {$p,+$};
\node[state] (d) [left=of c]{};
\node[state] (e) [right=of a] {$p_1,+$};
\node[state] (f) [above=of e] {$p_2,-$};
\node[state] (g) [right=of e] {$p_4,-$};
\node[state] (h) [right=of f] {$p_3,+$};
\path (a) edge node[below] {$b$} (b);
\path (a) edge node [right] {$a$} (c);
\path (c) edge node[above] {$b$} (d);
\path (e) edge node[left] {$a$} (f);
\path (e) edge node[below] {$b$} (g);
\path (f) edge node[above] {$b$} (h);
\path (g) edge node[right] {$a+b$} (h);
\end{tikzpicture}
\vfill
\caption{Graphs of weights for $M$ and $P\# Q(a,a+b)$}\label{fig6a}
\end{subfigure}\qquad
\begin{subfigure}[b][4.3cm][s]{.5\textwidth}
\centering
\vfill
\begin{tikzpicture}[state/.style ={circle, draw}]
\node[state] (a) {$p_1,-$};
\node[state] (b) [left=of a] {};
\node[state] (c) [above=of a] {$p_3,+$};
\node[state] (d) [left=of c]{};
\path (a) edge node[below] {$b$} (b);
\path (a) edge node [right] {$a+b$} (c);
\path (c) edge node[above] {$b$} (d);
\end{tikzpicture}
\vfill
\caption{Connected sum}\label{fig6b}
\end{subfigure}
\caption{Connected sum at the invariant spheres of $M$ and $P\# Q(a,a+b)$}\label{fig6}
\end{figure}
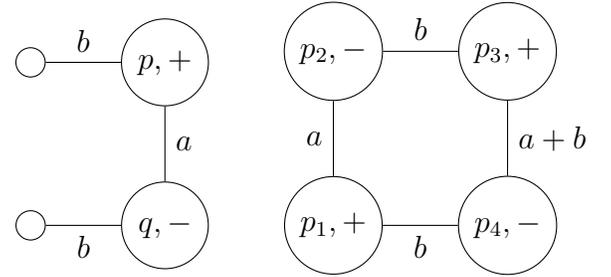
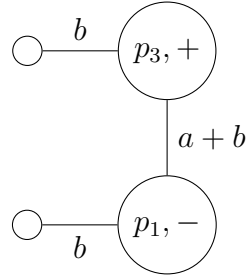

Let $M$ be a 4-dimensional compact oriented $S^1$-manifold. Suppose that two fixed points $p$ and $q$ both have weights $\{a,b\}$, $\varepsilon_p=+1$, $\varepsilon_q=-1$, and there is an invariant 2-sphere $S$ with weight $a$ between $p$ and $q$. 
Now, there is an orientation reversing equivariant diffeomorphism between an equivariant tubular neighborhod of the 2-sphere $S$ and an equivariant tubular neighborhood of the 2-sphere $S'$ of
$P\# Q(a,a+b)$ containing $p_1$ and $p_2$ of Example \ref{e23}. Thus, we can equivariantly glue $M$ and $P\# Q(a,a+b)$ along neighborhoods of $S$ and $S'$, to construct another $S^1$-manifold $N$. Figure \ref{fig6a} is a graph of weights for $M \sqcup P\# Q(a,a+b)$, and Figure \ref{fig6b} is a graph of weights for $N$.

The converse of the above equivariant connected sum is an equivariant splitting of $N$ into $M$ and $P \# Q(a,a+b)$; Let $N$ be a 4-dimensional compact oriented $S^1$-manifold and let $G_N$ be its graph of weights. Suppose that there is an invariant sphere of weight $a+b$ between two fixed points $p'$ and $q'$ such that $\varepsilon_{p'}=+1$, $\varepsilon_{q'}=-1$, and both $p'$ and $q'$ have weights $\{a+b,b\}$ for some positive integers $a$ and $b$; Figure \ref{fig6b} is a graph of weights for $N$ with $p_3=p'$ and $p_1=q'$. By the reverse of the above equivariant connected sum, we can equivariantly split $N$ into another 4-dimensional compact oriented $S^1$-manifold $M$ and $P \# Q(a,a+b)$, where Figure \ref{fig6a} is a graph of weights for $M \sqcup P \# Q(a,a+b)$.

\subsection{Reducing every weight to 1 via splitting}

Given a 4-dimensional compact oriented $S^1$-manifold, via equivariant splitting we can convert it to another $S^1$-manifold, which is minimal in the sense that weights are small.

\begin{theorem}\label{ReduceWeight} Let $M$ be a 4-dimensional compact oriented $S^1$-manifold.
Then we can decompose $M$ into an equivariant connected sum of $M_0$ and copies of $\pm \mathbb{CP}^2$ and $\Sigma_1$, where every weight in the normal bundle of any fixed component of $M_0$ is $1$. \end{theorem}

\begin{proof}
This theorem holds if the biggest weight over all fixed components of $M$ is $1$. 
Thus, suppose that the biggest weight $\ell$ is bigger than $1$.
Every weight in the normal bundle of a fixed surface is $1$, and thus any weight larger than $1$ occurs at isolated fixed points.
Let $G_W(M)$ be a graph of weights that satisfies the properties of Theorem \ref{graph}.

Pick an edge $e$ whose weight is $\ell$, and let $p_i$ and $p_j$ be the vertices of $e$. The fixed points $p_i$ and $p_j$ have weights $\{\ell,a\}$ and $\{\ell,b\}$, respectively, for some positive integers $a,b<\ell$. Because $\ell>1$, by Theorem \ref{n_e}, there is an invariant 2-sphere $S_e$ containing $p_i$ and $p_j$, and the Euler number of the normal bundle of $S_e$ is the Euler number $n_e$ of the edge, which is $n_e=\frac{\varepsilon_i a + \varepsilon_j b}{\ell}$; in particular, $n_e$ is an integer. Thus, one of the following two cases occurs.
\begin{enumerate}
\item $\varepsilon_i=\varepsilon_j$ and $a+b=\ell$.
\item $\varepsilon_i=-\varepsilon_j$ and $a=b$.
\end{enumerate}

\begin{figure}
\centering
\begin{subfigure}[b][4.6cm][s]{.3\textwidth}
\centering
\vfill
\begin{tikzpicture}[state/.style ={circle, draw}]
\node[state] (a)  {$p_i,+$};
\node[state] (b) [above=of a] {$p_j,+$};
\node[state] (c) [left=of a] {};
\node[state] (d) [left=of b] {};
\path (a) edge node[below] {$a$} (c);
\path (b) edge node[above] {$b$} (d);
\path (a) edge node[right] {$a+b$} (b);
\end{tikzpicture}
\vfill
\caption{Graph for $M$}\label{fig7a}
\vspace{\baselineskip}
\end{subfigure}\qquad
\begin{subfigure}[b][4.6cm][s]{.4\textwidth}
\centering
\vfill
\begin{tikzpicture}[state/.style ={circle, draw}]
\node[state] (a) {$+$};
\node[state] (b) [below left=of a] {};
\node[state] (c) [above left=of a] {};
\node[state] (d) [right=of a]{$-$};
\node[state] (e) [below right=of d] {$+$};
\node[state] (f) [above right=of d] {$+$};
\path (a) edge node[right] {$a$} (b);
\path (a) edge node [right] {$b$} (c);
\path (d) edge node[left] {$a$} (e);
\path (e) edge node [right] {$a+b$} (f);
\path (d) edge node [left] {$b$} (f);
\end{tikzpicture}
\vfill
\caption{Graphs for $M'$ and $P(a,a+b)$}\label{fig7b}
\vspace{\baselineskip}
\end{subfigure}\qquad
\caption{Equivariant splitting of $M$ into $M'$ and $P(a,a+b)$}\label{fig7}
\end{figure}
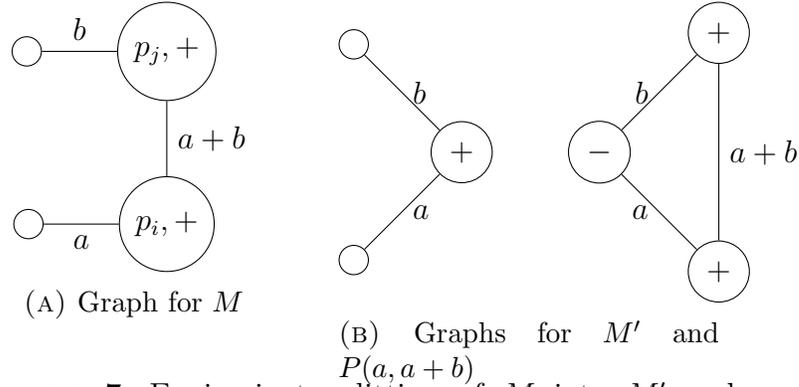

Assume that Case (1) holds. First, suppose that $\varepsilon_i=\varepsilon_j=+1$. Figure \ref{fig7a} is a graph of weights for $M$ near $p_i$ and $p_j$. As described in the third paragraph of Section \ref{s5.1}, we can split $M$ equivariantly into another 4-dimensional compact oriented $S^1$-manifold $M'$ and $P(a,a+b)$, whose graphs of weights are the left (denoted by $G_W(M')$) and right figures of Figure \ref{fig7b}, respectively.

\begin{figure}
\centering
\begin{subfigure}[b][4.6cm][s]{.3\textwidth}
\centering
\vfill
\begin{tikzpicture}[state/.style ={circle, draw}]
\node[state] (a)  {$p_i,-$};
\node[state] (b) [above=of a] {$p_j,-$};
\node[state] (c) [left=of a] {};
\node[state] (d) [left=of b] {};
\path (a) edge node[below] {$a$} (c);
\path (b) edge node[above] {$b$} (d);
\path (a) edge node[right] {$a+b$} (b);
\end{tikzpicture}
\vfill
\caption{Graph for $M$}\label{fig8a}
\vspace{\baselineskip}
\end{subfigure}\qquad
\begin{subfigure}[b][4.6cm][s]{.4\textwidth}
\centering
\vfill
\begin{tikzpicture}[state/.style ={circle, draw}]
\node[state] (a) {$+$};
\node[state] (b) [below left=of a] {};
\node[state] (c) [above left=of a] {};
\node[state] (d) [right=of a]{$+$};
\node[state] (e) [below right=of d] {$-$};
\node[state] (f) [above right=of d] {$-$};
\path (a) edge node[right] {$a$} (b);
\path (a) edge node [right] {$b$} (c);
\path (d) edge node[left] {$a$} (e);
\path (e) edge node [right] {$a+b$} (f);
\path (d) edge node [left] {$b$} (f);
\end{tikzpicture}
\vfill
\caption{Graphs for $M'$ and $Q(a,a+b)$}\label{fig8b}
\vspace{\baselineskip}
\end{subfigure}\qquad
\caption{Equivariant splitting of $M$ into $M'$ and $Q(a,a+b)$}\label{fig8}
\end{figure}
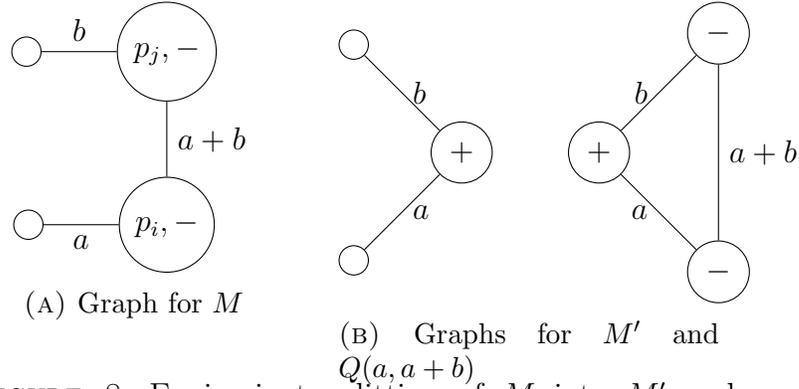

The case that $\varepsilon_i=\varepsilon_j=-1$ is analogous. Figure \ref{fig8a} is a graph of weights for $M$ near $p_i$ and $p_j$. As described in the fourth paragraph of Section \ref{s5.1}, we can split $M$ equivariantly into another 4-dimensional compact oriented $S^1$-manifold $M'$ and $Q(a,a+b)$, whose graphs of weights are the left (denoted by $G_W(M')$) and right figures of Figure \ref{fig8b}, respectively.

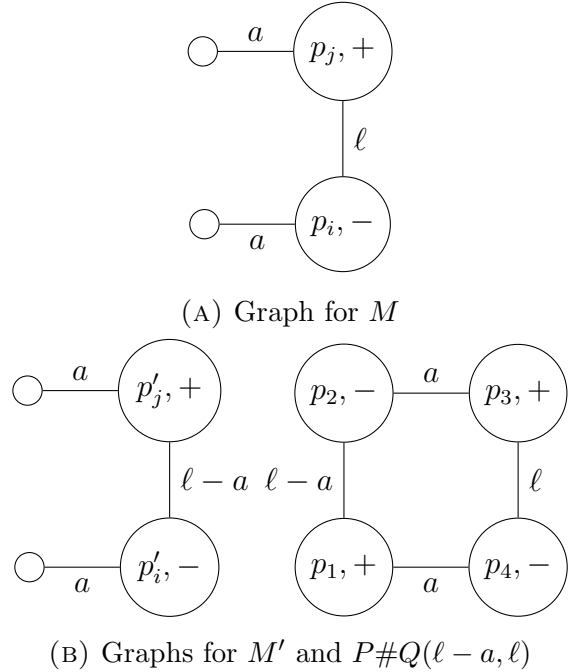
\begin{figure}
\centering
\begin{subfigure}[b][4.5cm][s]{.4\textwidth}
\centering
\vfill
\begin{tikzpicture}[state/.style ={circle, draw}]
\node[state] (a) {$p_i,-$};
\node[state] (b) [left=of a] {};
\node[state] (c) [above=of a] {$p_j,+$};
\node[state] (d) [left=of c]{};
\path (a) edge node[below] {$a$} (b);
\path (a) edge node [right] {$\ell$} (c);
\path (c) edge node[above] {$a$} (d);
\end{tikzpicture}
\vfill
\caption{Graph for $M$}\label{fig9a}
\end{subfigure}\qquad
\begin{subfigure}[b][4.5cm][s]{.6\textwidth}
\centering
\vfill
\begin{tikzpicture}[state/.style ={circle, draw}]
\node[state] (a) {$p_i',-$};
\node[state] (b) [left=of a] {};
\node[state] (c) [above=of a] {$p_j',+$};
\node[state] (d) [left=of c]{};
\node[state] (e) [right=of a] {$p_1,+$};
\node[state] (f) [above=of e] {$p_2,-$};
\node[state] (g) [right=of e] {$p_4,-$};
\node[state] (h) [right=of f] {$p_3,+$};
\path (a) edge node[below] {$a$} (b);
\path (a) edge node [right] {$\ell-a$} (c);
\path (c) edge node[above] {$a$} (d);
\path (e) edge node[left] {$\ell-a$} (f);
\path (e) edge node[below] {$a$} (g);
\path (f) edge node[above] {$a$} (h);
\path (g) edge node[right] {$\ell$} (h);
\end{tikzpicture}
\vfill
\caption{Graphs for $M'$ and $P \# Q(\ell-a,\ell)$}\label{fig9b}
\end{subfigure}\qquad
\caption{Equivariant splitting of $M$ into $M'$ and $P \# Q(\ell-a,\ell)$}\label{fig9}
\end{figure}

Assume that Case (2) holds. Figure \ref{fig9a} is a graph of weights for $M$ near $p_i$ and $p_j$. Without loss of generality, assume that $\varepsilon_i=-1=-\varepsilon_j$. As described in the last paragraph of Section \ref{s5.1}, we can split $M$ equivariantly into another 4-dimensional compact oriented $S^1$-manifold $M'$ and $P \# Q(\ell-a,\ell)$, whose graphs of weights are the left (denoted by $G_W(M')$) and right figures of Figure \ref{fig9b}, respectively.

We have that $G_W(M')$ is a graph of weights for $M'$. On $M'$ we consider the biggest weight, and repeat the above prodecure. Continuing the above argument, this theorem follows. \end{proof}

\subsection{Further: convert to one that has fixed surfaces only}

\begin{figure}
\centering
\begin{subfigure}[b][4cm][s]{.4\textwidth}
\centering
\vfill
\begin{tikzpicture}[state/.style ={circle, draw}]
\node[state] (a) {$+$};
\node[state] (b) [above=of a] {$S^2,1$};
\path (a) [bend right =20]edge node[right] {$1$} (b);
\path (b) [bend right =20]edge node[left] {$1$} (a);
\end{tikzpicture}
\vfill
\caption{$P(1,0)$}\label{fig10a}
\vspace{\baselineskip}
\end{subfigure}\qquad
\begin{subfigure}[b][4cm][s]{.4\textwidth}
\centering
\vfill
\begin{tikzpicture}[state/.style ={circle, draw}]
\node[state] (a) {$-$};
\node[state] (b) [above=of a] {$S^2,-1$};
\path (a) [bend right =20]edge node[right] {$1$} (b);
\path (b) [bend right =20]edge node[left] {$1$} (a);
\end{tikzpicture}
\vfill
\caption{$Q(1,0)$}\label{fig10b}
\vspace{\baselineskip}
\end{subfigure}\qquad
\caption{Graph for $P(1,0)$ and $Q(1,0)$}\label{fig10}
\end{figure}
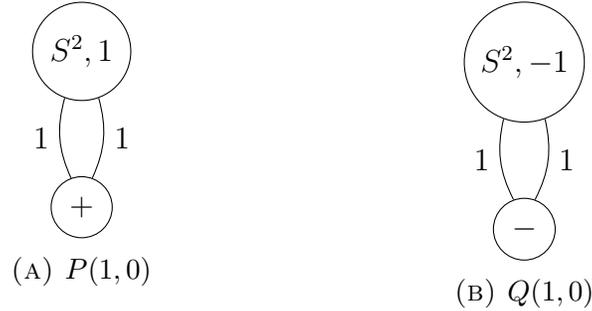

We can further convert any 4-dimensional compact oriented $S^1$-manifold into another $S^1$-manifold, which has only fixed surfaces.

\begin{theorem} \label{surface} Let the circle act on a 4-dimensional compact oriented manifold $M$. Then by taking equivariant splittings with $\pm \mathbb{CP}^2$, $\Sigma_1$ and by taking connected sums with $\pm \mathbb{CP}^2$, we can convert $M$ into another 4-dimensional compact oriented $S^1$-manifold $N$, which only has fixed surfaces.
\end{theorem}

\begin{proof}
By Theorem \ref{ReduceWeight}, we can decompose $M$ into an equivariant connected sum of $M_0$ and copies of $\pm \mathbb{CP}^2$, where every weight in the normal bundle of any fixed component of $M_0$ is 1.

Let $p$ be an isolated fixed point in $M_0$. It has weights $\{1,1\}$. Suppose that $\varepsilon_p=+1$. In this case, we can take an equivariant connected sum of $M_0$ and $Q(1,0)$  at $p$ and $[1:0:0]$ to construct another $S^1$-manifold $M'$, in which $p$ is replaced with a fixed 2-sphere. Figure \ref{fig11a} is a graph of weights for $M \sqcup Q(1,0)$ and Figure \ref{fig11b} is one for $M'$. 

If $\varepsilon_p=-1$, then we take an equivariant connected sum of $M_0$ and $P(1,0)$ at $p$ and $[1:0:0]$; Figure \ref{fig12a} is a graph of weights for $M \sqcup P(1,0)$ and Figure \ref{fig12b} is one for $M'$.

Continuing this to every isolated fixed point, the theorem follows.
\end{proof}

\begin{figure}
\centering
\begin{subfigure}[b][4.2cm][s]{.9\textwidth}
\centering
\vfill
\begin{tikzpicture}[state/.style ={circle, draw}]
\node[state] (a) {$p,+$};
\node[state] (b) [below left=of a] {};
\node[state] (c) [above left=of a] {};
\node[state] (e) [right=of a] {$p_1,-$};
\node[state] (f) [right=of e] {$S^2,-1$};
\path (a) edge node[below] {$1$} (b);
\path (a) edge node [right] {$1$} (c);
\path (e) [bend right =20]edge node[below] {$1$} (f);
\path (f) [bend right =20]edge node[above] {$1$} (e);
\end{tikzpicture}
\vfill
\caption{Graphs for $M_0$ and $Q(1,0)$}\label{fig11a}
\vspace{\baselineskip}
\end{subfigure}\qquad
\begin{subfigure}[b][4.2cm][s]{.9\textwidth}
\centering
\vfill
\begin{tikzpicture}[state/.style ={circle, draw}]
\node[state] (a) {$S^2,-1$};
\node[state] (b) [below left=of a] {};
\node[state] (c) [above left=of a] {};
\path (a) edge node[below] {$1$} (b);
\path (a) edge node [right] {$1$} (c);
\end{tikzpicture}
\vfill
\caption{Connected sum}\label{fig11b}
\vspace{\baselineskip}
\end{subfigure}\qquad
\caption{Connected sum of $M_0$ and $Q(1,0)$ at $p$ and $p_1$}\label{fig11}
\end{figure}
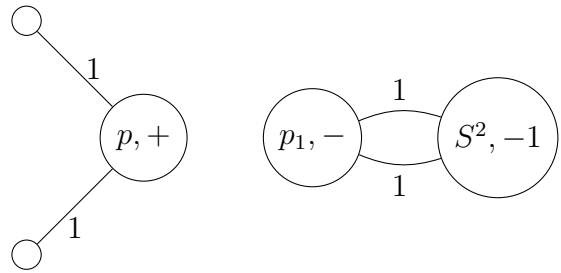
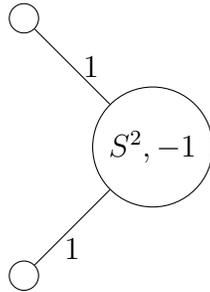

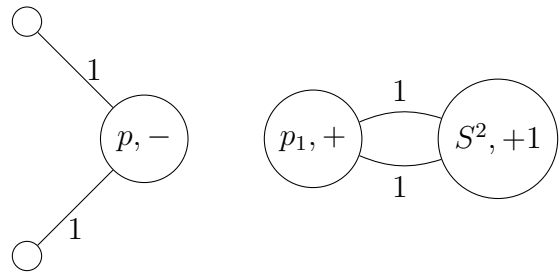
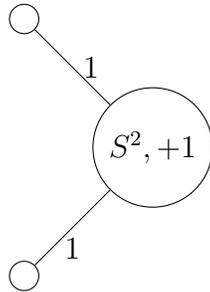
\begin{figure}
\centering
\begin{subfigure}[b][4.2cm][s]{.9\textwidth}
\centering
\vfill
\begin{tikzpicture}[state/.style ={circle, draw}]
\node[state] (a) {$p,-$};
\node[state] (b) [below left=of a] {};
\node[state] (c) [above left=of a] {};
\node[state] (e) [right=of a] {$p_1,+$};
\node[state] (f) [right=of e] {$S^2,+1$};
\path (a) edge node[below] {$1$} (b);
\path (a) edge node [right] {$1$} (c);
\path (e) [bend right =20]edge node[below] {$1$} (f);
\path (f) [bend right =20]edge node[above] {$1$} (e);
\end{tikzpicture}
\vfill
\caption{Graphs for $M_0$ and $P(1,0)$}\label{fig12a}
\vspace{\baselineskip}
\end{subfigure}\qquad
\begin{subfigure}[b][4.2cm][s]{.9\textwidth}
\centering
\vfill
\begin{tikzpicture}[state/.style ={circle, draw}]
\node[state] (a) {$S^2,+1$};
\node[state] (b) [below left=of a] {};
\node[state] (c) [above left=of a] {};
\path (a) edge node[below] {$1$} (b);
\path (a) edge node [right] {$1$} (c);
\end{tikzpicture}
\vfill
\caption{Connected sum}\label{fig12b}
\vspace{\baselineskip}
\end{subfigure}\qquad
\caption{Connected sum of $M_0$ and $P(1,0)$ at $p$ and $p_1$}\label{fig12}
\end{figure}

Note that by the second and third equations of Theorem \ref{3formulas}, the manifold $N$ satisfies
$$\sum_{i=1}^k n_j=0$$
and
$$L(N')=0.$$

Theorem \ref{surface} raises the following natural questions.

\begin{que}
Let $N$ be a 4-dimensional compact oriented $S^1$-manifold with fixed surfaces only. 
\begin{enumerate}
\item Is the normal bundle of any fixed surface of $N$ trivial?
\item If not, can we construct an example of such $N$ that has only fixed surfaces but some normal bundle of a fixed surface is not trivial?
\item Or, can we further convert $N$ to another $S^1$-manifold, in which the normal bundle of any fixed surface has Euler number zero?
\end{enumerate}
\end{que}

\section{Isolated fixed points} \label{s6}

\subsection{Converse to Theorems \ref{AHformula} and \ref{graph}}

In the case of isolated fixed points, we prove the converse of a combination of Theorem \ref{AHformula} and Theorem \ref{graph} that if an abstract graph of weights satisfies the Atiyah-Hirzebruch formula, then the graph is realized as one for a 4-dimensional oriented $S^1$-manifold.

\begin{theorem} \label{graphconverse}
Let $G_W$ be an admissible graph of weights that has points only. Suppose that $G_W$ satisfies the Atiyah-Hirzebruch formula, Equation \eqref{ASformula}.
Then there exists a 4-dimensional compact connected oriented $S^1$-manifold, whose graph of weights is $G_W$.
\end{theorem}

\begin{proof}
We shall mean that the graph $G_W$ satisfies the Atiyah-Hirzebruch formula (Equation \eqref{ASformula}) in the following sense: if $G_W$ has vertices $v_1,\cdots,v_m$ and a vertex $v_i$ has sign $\varepsilon_i$ and edges $e_{i1}$ and $e_{i2}$ with labels $w_{i1}$ and $w_{i2}$, respectively, then
\begin{equation} \label{eq4}
\displaystyle 
\sum_{i=1}^m \varepsilon_i\, {\frac{(1+z^{w_{i1}})(1+z^{w_{i2}})}{(1-z^{w_{i1}})(1-z^{w_{i2}})}}
\end{equation}
is a constant for all indeterminates $z$.

The idea of proof is analogous to that of Theorem \ref{ReduceWeight}. Let $e$ be an edge whose label $w_e$ is the biggest among the labels of the edges of $G_W$; let $p$ and $q$ be the vertices of $e$, and let $w$ ($w'$) be the label of the other edge of $p$ ($q$, respectively). If $w_e=1$ or $w_e=2$, then the theorem follows from Lemma \ref{weight1}. Thus, from now on, we assume that $w_e>2$. By the definition of a graph of weights, $w_e$ and $w$ are relatively prime and similarly for $w_e$ and $w'$. In particular, $w,w'<w_e$. Because $G_W$ is admissible, the Euler number $n_e=\frac{\varepsilon_p w+\varepsilon_q w'}{w_e}$ of $e$ is an integer. Thus, one of the following cases holds.

\begin{enumerate}
\item $\varepsilon_p=\varepsilon_q=+1$ and $w+w'=w_e$.
\item $\varepsilon_p=\varepsilon_q=-1$ and $w+w'=w_e$.
\item $\varepsilon_p=-\varepsilon_q$ and $w=w'$.
\end{enumerate}

Suppose that Case (1) holds. Assume that there exists a 4-dimensional compact connected oriented $S^1$-manifold $M$ whose graph of weights is $G_W$; see Figure \ref{fig13a}. As described in the third paragraph of Section \ref{s5.1}, by equivariant splitting, we can decompose $M$ into another 4-dimensional compact connected oriented $S^1$-manifold $M'$ and $P(w,w+w')$; their graphs are the left (denoted $G_W'$) and right figures of Figure \ref{fig13b}.

\begin{figure}
\centering
\begin{subfigure}[b][5.5cm][s]{.3\textwidth}
\centering
\vfill
\begin{tikzpicture}[state/.style ={circle, draw}]
\node[state] (a)  {$p,+$};
\node[state] (b) [above=of a] {$q,+$};
\node[state] (c) [left=of a] {};
\node[state] (d) [left=of b] {};
\path (a) edge node[below] {$w$} (c);
\path (b) edge node[above] {$w'$} (d);
\path (a) edge node[right] {$w+w'$} (b);
\end{tikzpicture}
\vfill
\caption{$G_W$}\label{fig13a}
\vspace{\baselineskip}
\end{subfigure}\qquad
\begin{subfigure}[b][5.5cm][s]{.6\textwidth}
\centering
\vfill
\begin{tikzpicture}[state/.style ={circle, draw}]
\node[state] (a) {$p',+$};
\node[state] (b) [below left=of a] {};
\node[state] (c) [above left=of a] {};
\node[state] (d) [right=of a]{$q_1,-$};
\node[state] (e) [below right=of d] {$q_0,+$};
\node[state] (f) [above right=of d] {$q_2,+$};
\path (a) edge node[right] {$w$} (b);
\path (a) edge node [right] {$w'$} (c);
\path (d) edge node[left] {$w$} (e);
\path (e) edge node [right] {$w+w'$} (f);
\path (d) edge node [left] {$w'$} (f);
\end{tikzpicture}
\vfill
\caption{$G_W'$ and graph of $P(w,w+w')$}\label{fig13b}
\vspace{\baselineskip}
\end{subfigure}\qquad
\caption{Splitting of $G_W$ into $G_W'$ and graph of  $P(w,w+w')$}\label{fig13}
\end{figure}
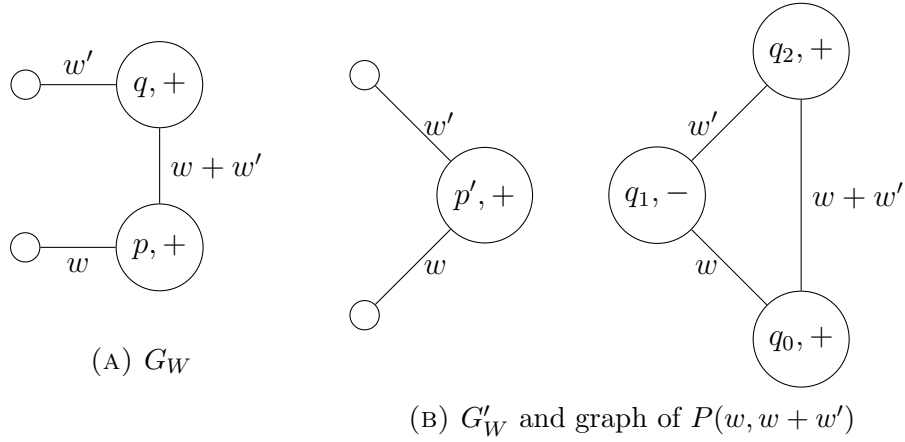

The right figure of Figure \ref{fig13b}, a graph of weights for $P(w,w+w')$, satisfies Equation \eqref{ASformula}; it is
\begin{equation} \label{eq5}
\displaystyle \frac{(1+z^w)(1+z^{w+w'})}{(1-z^w)(1-z^{w+w'})}-\frac{(1+z^w)(1+z^{w'})}{(1-z^w)(1-z^{w'})}+\frac{(1+z^{w'})(1+z^{w+w'})}{(1-z^{w'})(1-z^{w+w'})}=1,
\end{equation}
which is a constant for all indeterminates $z$.

We consider \eqref{ASformula} for $G_W'$;
\begin{equation} \label{eq6}
\displaystyle 
\sum_{v_i' \in V(G_W')}\varepsilon_i'\, {\frac{(1+z^{w_{i1}'})(1+z^{w_{i2}'})}{(1-z^{w_{i1}'})(1-z^{w_{i2}'})}}.
\end{equation}
The graph $G_W'$ is obtained from $G_W$ by shrinking the edge $e$ to a vertex $p'$, removing vertices $p$ and $q$. The vertex set of $G_W'$ minus $p'$ is equal to the vertex set of $G_W$ minus $p$ and $q$. In \eqref{eq6}, the vertex $p'$ of $G_W'$ contributes the term 
\begin{center}
$\displaystyle \frac{(1+z^{w})(1+z^{w'})}{(1-z^{w})(1-z^{w'})}$.
\end{center}
The sum of this term with \eqref{eq5} is
\begin{center}
$\displaystyle \frac{(1+z^w)(1+z^{w+w'})}{(1-z^w)(1-z^{w+w'})}+\frac{(1+z^{w'})(1+z^{w+w'})}{(1-z^{w'})(1-z^{w+w'})},$
\end{center}
which are precisely the terms that $p$ and $q$ contribute in \eqref{eq4}. In other words, the sum of \eqref{eq5} and \eqref{eq6} is exactly \eqref{eq4}. This means that if $G_W$ satisfies Equation \eqref{ASformula}, then so does $G_W'$, and vice versa. The manifold $M$ splits equivariantly into $M'$ and $P(w,w+w')$, and $M$ is an equivariant connected sum of $M'$ and $P(w,w+w')$. Therefore, the existence of such a manifold $M$ for $G_W$ is equivalent to the existence of $M'$ for $G_W'$, which has one less edge with label $w_e$. Moreover, $G_W$ is admissible if and only if $G_W'$ is admissible.

Case (2) is completely analogous. In this case, the labels of edges of $p$ and $q$ are the same as Case (1) but their signs are reversed. Thus instead of $P(w,w+w')$ and its graph (Figure \ref{fig1-2}), we use $Q(w,w+w')$ and its graph (Figure \ref{fig1-3}).

Suppose that Case (3) holds. Assume that $\varepsilon_p=-\varepsilon_q=+1$; the other case $\varepsilon_p=-\varepsilon_q=-1$ is analogous. Assume that there exists a 4-dimensional compact connected oriented $S^1$-manifold $M$ whose graph of weights is $G_W$; see Figure \ref{fig14a}. As described in the last paragraph of Section \ref{s5.1}, by equivariant splitting, we can decompose $M$ into another 4-dimensional compact connected oriented $S^1$-manifold $M'$ and $P \# Q(w_e-w,w)$; their graphs are the left (denoted $G_W'$) and right figures of Figure \ref{fig14b}.

\begin{figure}
\centering
\begin{subfigure}[b][4.5cm][s]{.4\textwidth}
\centering
\vfill
\begin{tikzpicture}[state/.style ={circle, draw}]
\node[state] (a) {$q,-$};
\node[state] (b) [left=of a] {};
\node[state] (c) [above=of a] {$p,+$};
\node[state] (d) [left=of c]{};
\path (a) edge node[below] {$w$} (b);
\path (a) edge node [right] {$w_e$} (c);
\path (c) edge node[above] {$w$} (d);
\end{tikzpicture}
\vfill
\caption{$G_W$}\label{fig14a}
\end{subfigure}\qquad
\begin{subfigure}[b][4.5cm][s]{.6\textwidth}
\centering
\vfill
\begin{tikzpicture}[state/.style ={circle, draw}]
\node[state] (a) {$q',-$};
\node[state] (b) [left=of a] {};
\node[state] (c) [above=of a] {$p',+$};
\node[state] (d) [left=of c]{};
\node[state] (e) [right=of a] {$p_1,+$};
\node[state] (f) [above=of e] {$p_2,-$};
\node[state] (g) [right=of e] {$p_4,-$};
\node[state] (h) [right=of f] {$p_3,+$};
\path (a) edge node[below] {$w$} (b);
\path (a) edge node [left] {$w_e-w$} (c);
\path (c) edge node[above] {$w$} (d);
\path (e) edge node[left] {$w_e-w$} (f);
\path (e) edge node[below] {$w$} (g);
\path (f) edge node[above] {$w$} (h);
\path (g) edge node[right] {$w_e$} (h);
\end{tikzpicture}
\vfill
\caption{$G_W'$ and graph of $P \# Q(w_e-w,w)$}\label{fig14b}
\end{subfigure}\qquad
\caption{Splitting of $G_W$ into $G_W'$ and graph of $P \# Q(w_e-w,w)$}\label{fig14}
\end{figure}
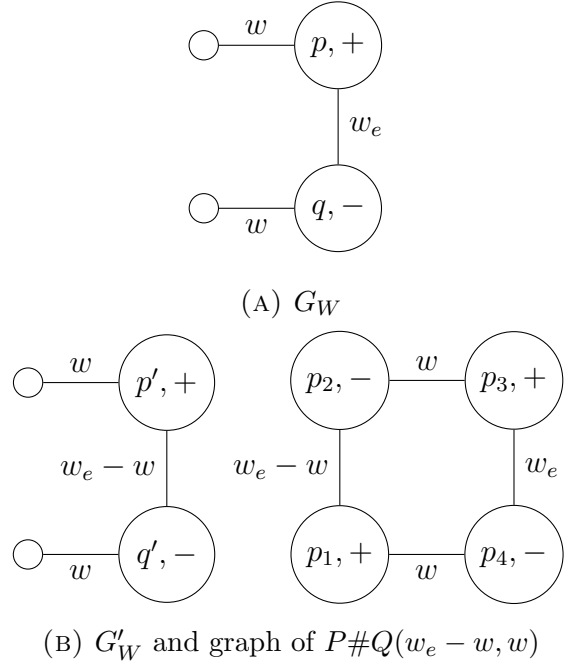

The right figure of Figure \ref{fig14b}, which is a graph of weights for $P \# Q(w_e-w,w)$, satisfies Equation \eqref{ASformula}; it is
\begin{multline} \label{eq7}
\displaystyle -\frac{(1+z^w)(1+z^{w_e})}{(1-z^w)(1-z^{w_e})}+\frac{(1+z^w)(1+z^{w_e-w})}{(1-z^w)(1-z^{w_e-w})}\\+\frac{(1+z^{w})(1+z^{w_e})}{(1-z^{w})(1-z^{w_e})}-\frac{(1+z^{w})(1+z^{w_e-w})}{(1-z^{w})(1-z^{w_e-w})}=0,
\end{multline}
which is a constant for all indeterminates $z$.

We consider \eqref{ASformula} for $G_W'$;
\begin{equation} \label{eq8}
\displaystyle 
\sum_{v_i' \in V(G_W')}\varepsilon_i'\, {\frac{(1+z^{w_{i1}'})(1+z^{w_{i2}'})}{(1-z^{w_{i1}'})(1-z^{w_{i2}'})}}.
\end{equation}
The graph $G_W'$ is obtained from $G_W$ by replacing the label $w_e$ of the edge $e$ with $w_e-w$. The vertex set of $G_W'$ is equal to the vertex set of $G_W$. In \eqref{eq8}, the vertices $p'$ and $q'$ (that are $p$ and $q$ in $G_W$) of $G_W'$ contribute the terms
\begin{center}
$\displaystyle \frac{(1+z^w)(1+z^{w_e-w})}{(1-z^w)(1-z^{w_e-w})}-\frac{(1+z^{w})(1+z^{w_e-w})}{(1-z^{w})(1-z^{w_e-w})}$.
\end{center}
The sum of these terms with \eqref{eq7} is
\begin{center}
$\displaystyle \frac{(1+z^w)(1+z^{w_e})}{(1-z^w)(1-z^{w_e})}-\frac{(1+z^w)(1+z^{w_e})}{(1-z^w)(1-z^{w_e})},$
\end{center}
which are precisely the terms that $p$ and $q$ contribute in \eqref{eq4}. In other words, the sum of \eqref{eq7} and \eqref{eq8} is exactly \eqref{eq4}. This means that if $G_W$ satisfies Equation \eqref{ASformula}, then so does $G_W'$, and vice versa. The manifold $M$ splits equivariantly into $M'$ and $P \# Q(w_e-w,w)$, and $M$ is an equivariant connected sum of $M'$ and $P \# Q(w_e-w,w)$. Therefore, the existence of such a manifold $M$ for $G_W$ is equivalent to the existence of $M'$ for $G_W'$, which has one less edge with label $w_e$. Moreover, $G_W$ is admissible if and only if $G_W'$ is admissible.

We repeat the above procedure. Assume that an admissible graph $G_W'$ of weights satisfies \eqref{ASformula}, consider the edge whose label is the biggest, and so on. Repeating the above, the existence of a 4-dimensional compact connected oriented $S^1$-manifold whose graph of weights is $G_W$, reduces to the existence of another 4-dimensional compact connected oriented $S^1$-manifold $M''$ whose graph $G_W''$ of weights satisfies Equation \eqref{ASformula} and only has edges of label 1. By Lemma \ref{weight1} below, such a manifold $M''$ exists.
\end{proof}

\begin{lemma} \label{weight1}
Let $G_W$ be a graph of weights that has points only, satisfies \eqref{ASformula}, and the label of every edge of $G_W$ is 1. Then there exists a 4-dimensional compact connected oriented $S^1$-manifold $M$, whose graph of weights is $G_W$.
\end{lemma}

\begin{proof}
Let $v_1,\cdots,v_m$ be the vertices of $G_W$, and let $\varepsilon_i$ be the sign of $v_i$. By the assumption, $G_W$ satisfies \eqref{ASformula}; because every label is 1, it is
\begin{center}
$\displaystyle \sum_{i=1}^m \varepsilon_i\, {\frac{(1+z)(1+z)}{(1-z)(1-z)}}=\sum_{i=1}^n \varepsilon_i \, \{(1+z)(1+z+z^2+\cdots)\}^2=\sum_{i=1}^n \varepsilon_i \,  (1+2 \sum_{j=1}^\infty z^j)^2$,
\end{center}
and is a constant for all indeterminates $z$. For the formula to be a constant for all indeterminates $z$, we must have that the formula is identically zero and the number $k$ of vertices with sign $+1$ is equal to the number of vertices with sign $-1$.

Let $M$ be an equivariant connected sum along free orbits of $k$-copies of $S(1,1)$ of Example \ref{S4}. Each $S(1,1)$ has two fixed points $(0,0,1)$ with sign $+1$ and $(0,0,-1)$ with sign $-1$ that have weights $\{1,1\}$. Thus, $M$ has $k$ fixed points with sign $+1$ and weights $\{1,1\}$ and $k$ fixed points with sign $-1$ and weights $\{1,1\}$. By Definition \ref{graphofweights}, $G_W$ is a graph of weights of $M$. \end{proof}

\subsection{Minimal manifold's graph has vanishing Euler number}

For a 4-dimensional oriented $S^1$-manifold has isolated fixed points, its minimal manifold admits a graph of weights whose edges have vanishing Euler numbers.

\begin{cor} \label{n=0}
Let the circle act on a 4-dimensional compact oriented manifold $M$ with isolated fixed points. Then we can decompose $M$ into an equivariant connected sum of $M_0$ and copies of $\pm \mathbb{CP}^2$, where every weight at a fixed point of $M_0$ is 1, and $M_0$ admits a graph $G_W$ of weights for which $n_e=0$ for all edges $e$ of $G_W$. 
\end{cor}

\begin{proof}
Because $M$ has isolated fixed points only, by Theorem \ref{ReduceWeight}, we can decompose $M$ into an equivariant connected sum of $M_0$ and copies of $\pm \mathbb{CP}^2$, where every weight at a fixed point of $M_0$ is 1, and the proof of Theorem \ref{ReduceWeight} implies that $M_0$ also has only isolated fixed points.

Let $G_W$ be a graph of weights of $M_0$ that satisfies the properties of Theorem \ref{graph}. Let $e$ be any edge of $G_W$. Let $p_i$ and $p_j$ be the vertices of $e$. By the property of Theorem \ref{graph}, $\varepsilon_i=-\varepsilon_j$. Since both fixed points $p_i$ and $p_j$ have weights $\{1,1\}$, by Definition \ref{OtherWeight}, the other weight of $w_e$ at $p_i$ ($p_j$) is 1. By Definition \ref{EulNum}, the Euler number of $e$ is
\begin{center}
$\displaystyle n_e=\frac{\varepsilon_i \cdot 1 + \varepsilon_j \cdot 1}{1}=0$. 
\end{center}
\end{proof}


 
 
 



 
 
 



\end{document}